\RequirePackage{fix-cm}
\documentclass[smallextended]{svjour3}       
\smartqed  
\usepackage{graphicx}
\usepackage{subfigure}
\usepackage[T1]{fontenc}
\usepackage{geometry}
\usepackage{algorithm}
\usepackage{algorithmic}
\usepackage{multirow}
\usepackage{url}
\usepackage{amsbsy,amsmath,latexsym,amsfonts, epsfig, color, authblk, amssymb, graphics, bm, caption}
\usepackage{epsf,slidesec,epic,eepic}
\usepackage{fancybox}
\usepackage{fancyhdr}
\usepackage[figuresright]{rotating}
\usepackage{makecell,booktabs}
\usepackage{setspace}
\usepackage{nccmath}
\usepackage{cases}
\usepackage{epstopdf}
\begin{document}

\title{An inexact PAM method for computing Wasserstein barycenter with unknown supports
\thanks{This work is supported by the National Natural Science Foundation of China
 under project No.11971177 and Guangdong Basic and Applied Basic Research Foundation
(2020A1515010408).}}

\titlerunning{Inexact PAM method for barycenter with unknown supports}

\author{Yitian Qian \and Shaohua Pan}


\institute{Yitian Qian \at
          School of Mathematics, South China University of Technology, Guangzhou.
              \email{mayttqian@mail.scut.edu.cn}
           \and
           Shaohua Pan\at
           School of Mathematics, South China University of Technology, Guangzhou.
           \email{shhpan@scut.edu.cn}
}

\date{Received: date / Accepted: date}

\maketitle

 \begin{abstract}
  Wasserstein barycenter is the centroid of a collection of discrete probability
  distributions which minimizes the average of the $\ell_2$-Wasserstein distance.
  This paper focuses on the computation of Wasserstein barycenters under
  the case where the support points are free, which is known to be a severe
  bottleneck in the D2-clustering due to the large-scale and nonconvexity.
  We develop an inexact proximal alternating minimization (iPAM) method for
  computing an approximate Wasserstein barycenter, and provide its global
  convergence analysis. This method can achieve a good accuracy with a reduced
  computational cost when the unknown support points of the barycenter have
  low cardinality. Numerical comparisons with the 3-block B-ADMM in \cite{YeWWL17}
  and an alternating minimization method involving the LP subproblems
  on synthetic and real data show that the proposed iPAM can yield comparable
  even a little better objective values in less CPU time, and hence
  the computed barycenter will render a better role in the D2-clustering.
\end{abstract}

 \section{Introduction}\label{intro}

  In machine learning, many complex instances such as images, sequences,
 and documents can be described in terms of discrete probability distributions.
 For example, the bag of ``words'' data model used in multimedia retrieval and
 document analysis is a discrete distribution, while the widely used normalized histogram
 which contains the fixed supports and associated weights is also a special case of
 discrete distributions. In this paper, we focus on the computation of
 the centroid of discrete probability distributions under the Wasserstein distance
 (also known as the Mallows distance \cite{Mallows72} or the earth mover's distance \cite{Rubner00}),
 which is called the Wasserstein barycenter.

 Wasserstein barycenters are often used to cluster the discrete probability distributions
 in D2-clustering, which minimizes the total within-cluster variation under the Wasserstein distance
 similarly as Lloyd's K-means for vectors under the Euclidean distance. This clustering problem
 was originally explored by Li and Wang \cite{Li08} who coined the phrase D2-clustering.
 The motivations for using the Wasserstein distance in practice are strongly argued
 by some researchers in the literature (see, e.g., \cite{Li08,Cui13,Cui14,Pele09}),
 and its theoretical significance is well supported in the optimal transport literature \cite{Villani08}.
 In the D2-clustering framework, the Wasserstein barycenter is required for
 the case of unknown supports with a pre-given cardinality. To compute it, one faces
 a large-scale nonconvex optimization problem in which a coupled nonconvex objective
 function is minimized over a polyhedral set. Due to the advantages of Wasserstein distance,
 D2-clustering holds much promise, but the high computational cost of barycenter
 has limited its applications.

 To scale up the computation of Wasserstein barycenter in the D2-clustering,
 a divide-and-conquer approach has been proposed in \cite{Zhang15}, but the method
 is ad-hoc and lack of convergence guarantee. When an alternating minimization
 strategy is used, the computation of Wasserstein barycenter is decomposed into
 a quadratic program with a closed form solution and a linear program (LP) with
 a super-linear time-complexity on the number of samples $N$. The latter brings
 a big challenge for the computation of a barycenter because the number of variables
 in the LP grows quickly with the number of support points, and the classical LP solvers
 such as the simplex method and the interior point method are not scalable
 (see Figure \ref{fig3}, \ref{fig5} and \ref{fig7}). To overcome this challenge,
 Ye and Li \cite{YeLi14} applied the classical alternating direction method of multipliers
 (ADMM) for solving this LP, Wang and Banerjee \cite{Wang14} generalized the classical
 ADMM to the Bregman ADMM (B-ADMM) by replacing the quadratic distance with a general
 Bregman distance so as to exploit the structure of the LP, and Yang et al. \cite{Yang18}
 recently proposed a very efficient dual solver for the LP by adopting a symmetric
 Gauss-Seidel based ADMM (sGS-ADMM). However, they neither studied the performance
 of an alternating minimization method with such a subroutine for computing barycenter
 nor provided the global convergence analysis for the outer alternating minimization method.
 Recently, Ye et al. \cite{YeWWL17} proposed a 3-block B-ADMM for computing a Wasserstein barycenter directly.
 Although this method has demonstrated a computational efficiency for large-scale data,
 it is still unclear whether it is globally convergent or not. In fact, for convex programs,
 it has been shown in \cite{Chen16} that the direct extension of the classical ADMM to
 the three-block case can be divergent.

 The main contribution of this work is to develop a globally convergent and efficient
 inexact proximal alternating minimization (iPAM) method for computing
 an approximate Wasserstein barycenter when the support points are unknown.
 Since a proximal alternating minimization strategy is used,
 each iteration involves two strongly convex quadratic programs (QPs).
 One of them has a closed form solution, and the other has a polyhedral constraint set
 and may be good-conditioned by controlling the proximal parameter elaborately.
 The strongly convex QPs have much better stability than those LPs appearing
 in \cite{YeLi14,Wang14}, which means that their solutions are much easier to achieve.
 In Section \ref{sec5.1}, we propose a tailored linearized ADMM for solving
 the strongly convex QP by exploiting the special structure of the feasible set.
 Different from the sGS-ADMM proposed in \cite{Yang18}, the linearized
 ADMM is a primal solver for a 2-block strongly convex QP instead of a dual solver
 for the 3-block LP. We notice that the B-ADMM in \cite{Wang14} also belongs to
 this line since at each iteration it transforms the LP into two simple strongly convex
 Kullback-Leibler minimization problems to solve. Compared with the B-ADMM,
 our linearized ADMM not only has a global convergence \cite{FPST13}
 but also admits the well-established linear rate of convergence \cite{HanSZ17} and
 weighted iteration complexity \cite{ShenPan16}. Numerical comparisons with
 the 3-block B-ADMM in \cite{YeWWL17} and an alternating minimization method
 involving the LP subproblems (ALMLP for short) on synthetic and real data indicate that
 the proposed iPAM yields comparable even a little better objective values
 within less computing time, and hence the computed approximate barycenter will
 render a better role in the D2-clustering.

 For a nonconvex and nonsmooth optimization problem, without additional conditions
 imposed on the problem, the convergence result of an alternating minimization method
 and more general block coordinate descent methods
 (see, e.g., \cite{Tseng01,Tseng09,Razaviyayn13,Hong16}) is typically limited to
 the objective value convergence (to a possibly non-minimal value) or the convergence
 of a certain subsequence of iterates to a critical point. Motivated by the recent
 excellent works \cite{Attouch10,Attouch13,Bolte14}, we achieve the global convergence
 of our iPAM method for computing a Wasserstein barycenter by using
 the Kurdyka-Lojasiewicz (KL) property of the extended objective function. It is worthwhile
 to point out that under the KL assumption Xu and Yin \cite{Xu17} also developed a globally
 convergent algorithm based on block coordinate update for a class of nonconvex and nonsmooth
 optimization problems.

 The rest of this paper is organized as follows. Section \ref{sec2} presents
 the notation and preliminary knowledge that will be used in this paper.
 In Section 3, we introduce the optimization model for computing a Wasserstein barycenter
 when its support points are unknown and propose an inexact PAM method
 for solving it. Section 4 focuses on the convergence analysis of the iPAM method.
 In Section 5, we provide the implementation details of the iPAM and compare its performance
 with that of the B-ADMM \cite{YeWWL17} and ALMLP on synthetic and real data.

 \section{Notation and preliminaries}\label{sec2}

 Throughout this paper, $\mathbb{R}^{m\times n}$ represents the vector space consisting of
 all $m\times n$ real matrices, equipped with the trace inner product $\langle \cdot,\cdot\rangle$
 and its induced Frobenius norm $\|\cdot\|_F$, i.e., $\langle X,Y\rangle={\rm tr}(X^{\mathbb{T}}Y)$
 for $X,Y\in\mathbb{R}^{m\times n}$, and $\mathbb{R}_{+}^{m\times n}$ denotes
 the polyhedral cone consisting of all $m\times n$ nonnegative real matrices.
 For a given vector $x\in\mathbb{R}^n$, $\|x\|$ means the Euclidean-norm
 in $\mathbb{R}^n$. For a given set $S$, $\delta_{S}$ denotes the indicator
 function on $S$, i.e., $\delta_{S}(u)=0$ if $u\in S$, otherwise $\delta_{S}(u)=+\infty$;
 when $S$ is convex, $\mathcal{N}_S(x)$ denotes the normal cone of $S$ at $x$
 in the sense of convex analysis \cite{Roc70}. The notation $e_p$
 denotes a column vector of all ones whose dimension is $p$.

 In the rest of this section, the notation $\mathbb{X}$ denotes a finite dimensional
 vector space equipped with the inner product $\langle \cdot,\cdot\rangle$ and
 its induced norm $\|\cdot\|$. First of all, we recall from \cite[Chapter 8]{RW98}
 the notion of generalized subdifferentials for an extended real-valued function.
 \subsection{Generalized subdifferential and critical point}\label{sec2.1}
 \begin{definition}\label{gsubdiff}
  Consider a function $h\!:\mathbb{X}\to(-\infty,+\infty]$
  and a point $\overline{x}$ with $h(\overline{x})$ finite.
  The regular subdifferential of $h$ at $x$, denoted by
  $\widehat{\partial}h(x)$, is defined as
  \[
    \widehat{\partial}h(x):=\bigg\{v\in\mathbb{R}^p\ \big|\
    \liminf_{x'\to x\atop x'\ne x}
    \frac{h(x')-h(x)-\langle v,x'-x\rangle}{\|x'-x\|}\ge 0\bigg\};
  \]
  and the (limiting) subdifferential of $h$ at $x$, denoted by $\partial h(x)$, is defined as
  \[
    \partial h(x):=\Big\{v\in\mathbb{R}^p\ |\  \exists\,x^k\xrightarrow[h]{}x\ {\rm and}\
    v^k\in\widehat{\partial}h(x^k)\to v\ {\rm as}\ k\to\infty\Big\}.
  \]
 \end{definition}
  \begin{remark}\label{remark-Fsubdiff}
   {\bf(a)} Notice that $\widehat{\partial}h(\overline{x})\subseteq
   \partial h(\overline{x})$, and the former is a closed convex set,
   but the latter is generally not convex.
   When $h$ is convex, $\partial h(\overline{x})=\widehat{\partial}h(\overline{x})$,
   which also coincides with the subdifferential of $h$ at $\overline{x}$
   in the sense of convex analysis \cite{Roc70}.

   \noindent
   {\bf(b)} Let $\{(x^k,\xi^k)\}$ be a sequence in the graph of the set-valued mapping
   $\partial h\!:\mathbb{X}\rightrightarrows\mathbb{X}$, which converges to
   $(\overline{x},\overline{\xi})$. If $h(x^k)\to h(\overline{x})$ as $k\to\infty$,
   then $(\overline{x},\overline{\xi})\in{\rm gph}\partial h$.

   \noindent
   {\bf(c)} By \cite[Theorem 10.1]{RW98}, a necessary condition
   for $\overline{x}\in\mathbb{X}$ to be a local minimizer of $h$ is
   $0\in\widehat{\partial}h(\overline{x})\subseteq\partial h(\overline{x})$.
   A point $x^*$ satisfying $0\in\partial h(x^*)$ (respectively, $0\in\widehat{\partial} h(x^*)$)
   is called a critical (respectively, regular critical) point of $h$. The critical point set
   of $h$ is denoted by ${\rm crit}\,h$.
  \end{remark}

  For an optimization problem with a nonconvex objective function but
  a closed convex feasible set, it is common to consider its directional
  stationary point, which is defined as follows.
  \begin{definition}\label{Dir-point}
   Let $h\!:\mathbb{X}\to\mathbb{R}$ be a directionally differentiable function,
   and $S\subseteq\mathbb{X}$ be a closed convex set. Consider a point
   $\overline{x}\in S$. If for every $x\in S$,
   \(
     h'(\overline{x};x-\overline{x})\ge 0,
   \)
   then $\overline{x}$ is called a directional stationary point of
   the minimization problem $\min_{x\in S}h(x)$.
 \end{definition}

 The following lemma states that for the problem in Definition \ref{Dir-point},
 when $h$ is locally Lipschitz relative to $S$ and regular, its directional
 stationary points are same as the critical point of $h+\delta_S$.
 \begin{lemma}\label{relation}
  Consider the minimization problem $\min_{x\in S}h(x)$ where
  $S\subseteq\mathbb{X}$ is a closed convex set.
  Suppose that $h\!:\mathbb{X}\to\mathbb{R}$ is directionally differentiable and locally Lipschitz
  at $\widehat{x}\in S$. If $\widehat{x}$ is a directional stationary point
  of this problem, then $0\in\partial h(\widehat{x})+\mathcal{N}_S(\widehat{x})$.
  If in addition $\partial h(\widehat{x})=\widehat{\partial}h(\widehat{x})$,
  then every critical point of $h+\delta_{S}$ is a directional stationary point
  of this problem.
 \end{lemma}
 \begin{proof}
  Write $f(x):=h(x)+\delta_{S}(x)$ for $x\in\mathbb{X}$.
 From the definition of subderivative and the directional differentiability of $h$,
 it is not hard to obtain that
 \begin{equation*}
   df(\widehat{x})(w)=h'(\widehat{x};w)+\delta_{\mathcal{T}_{S}(\widehat{x})}(w)
   \quad\forall w\in\mathbb{X}.
 \end{equation*}
 Since $h'(\widehat{x};x-\widehat{x})\ge 0$ for all $x\in S$,
 from the definition of radial cone, it follows that
 $h'(\widehat{x};w)\ge 0$ for all $w\in\mathcal{R}_{S}(\widehat{x})$.
 Notice that $\mathcal{T}_{S}(\widehat{x})={\rm cl}[\mathcal{R}_{S}(\widehat{x})]$.
 From the globally Lipschitz continuity of $h'(\widehat{x};\cdot)$,
 we have $h'(\widehat{x};w)\ge 0$ for all $w\in\mathcal{T}_{S}(\widehat{x})$.
 Together with the last equation, it follows that
 \[
   df(\widehat{x})(w)\ge 0\quad{\rm for\ all}\ w\in\mathbb{X}.
 \]
 Now pick any $(u,\alpha)\in\mathcal{T}_{{\rm epi}f}(\widehat{x},f(\widehat{x}))
 ={\rm epi}\,df(\widehat{x})$ where the equality is due to \cite[Theorem 8.2]{RW98}.
 Then,
 \(
   \langle (0,-1),(u,\alpha)\rangle=-\alpha\le-df(\widehat{x})(u)\le 0,
 \)
 which implies that
 \[
  (0,-1)\in[\mathcal{T}_{{\rm epi}f}(\widehat{x},f(\widehat{x}))]^{\circ}
 =\widehat{\mathcal{N}}_{{\rm epi}f}(\widehat{x},f(\widehat{x}))
 \subseteq \mathcal{N}_{{\rm epi}f}(\widehat{x},f(\widehat{x})).
 \]
 By \cite[Theorem 8.9]{RW98}, it follows that
 $0\in\partial f(\widehat{x})\subseteq \partial h(\widehat{x})+\mathcal{N}_{S}(\widehat{x})$.

 Now suppose that $\partial h(\widehat{x})=\widehat{\partial}h(\widehat{x})$.
 From $0\in \partial h(\widehat{x})+\mathcal{N}_{S}(\widehat{x})$,
 we have $0\in\widehat{\partial}h(\widehat{x})+\mathcal{N}_{S}(\widehat{x})$.
 Then, there exists $s\in\mathcal{N}_{S}(\widehat{x})$ such that
 $-s\in\widehat{\partial}h(\widehat{x})$. By \cite[Exercise 8.4]{RW98},
 for any $x\in S$,
 \[
   dh(\widehat{x})(x-\widehat{x})\ge\langle -s,x-\widehat{x}\rangle\ge 0.
 \]
 Notice that $dh(\widehat{x})(x-\widehat{x})=h'(\widehat{x};x-\widehat{x})$.
 So, $h'(\widehat{x};x-\widehat{x})\ge 0$ for all $x\in S$. \qed
 \end{proof}
 \subsection{Kurdyka-Lojasiewicz property}\label{subsec2.2}
 \begin{definition}\label{KL-def}
  Let $h\!:\mathbb{X}\to(-\infty,+\infty]$ be a proper lower semicontinuous (lsc) function.
  The function $h$ is said to have the Kurdyka-Lojasiewicz (KL) property
  at $\overline{x}\in{\rm dom}\,\partial h$ if there exist $\eta\in(0,+\infty]$,
  a continuous concave function $\varphi\!:[0,\eta)\to\mathbb{R}_{+}$ satisfying
  \begin{itemize}
  \item [(i)] $\varphi(0)=0$ and $\varphi$ is continuously differentiable on $(0,\eta)$,

  \item[(ii)] for all $s\in(0,\eta)$, $\varphi'(s)>0$;
  \end{itemize}
  and a neighborhood $\mathcal{U}$ of $\overline{x}$ such that for all
  \(
    x\in\mathcal{U}\cap\big[h(\overline{x})<h(x)<h(\overline{x})+\eta\big],
  \)
  \[
    \varphi'(h(x)-h(\overline{x})){\rm dist}(0,\partial h(x))\ge 1.
  \]
  If $h$ satisfies the KL property at each point of ${\rm dom}\,\partial h$,
  then $h$ is called a KL function.
 \end{definition}
 \begin{remark}\label{KL-remark}
  By Definition \ref{KL-def} and \cite[Lemma 2.1]{Attouch10},
  a proper lsc function has the KL property at every noncritical point.
  Thus, to show that a proper lsc $h\!:\mathbb{X}\to(-\infty,+\infty]$
  is a KL function, it suffices to check that $h$ has the KL property at
  any critical point.
 \end{remark}

 As discussed in \cite[Section 4]{Attouch10}, many classes of functions are the KL function;
 for example, the semialgebraic function. A function $h\!:\mathbb{R}^n\to(-\infty,+\infty]$
 is semialgebraic if its graph is a semialgebraic subset of $\mathbb{R}^{n+1}$.
 Recall that a subset of $\mathbb{R}^{n}$ is called semialgebraic if
 it can be written as a finite union of sets of the form
 \[
  \Omega=\bigcup_{i=1}^p\bigcap_{j=1}^q\Big\{x\in\mathbb{R}^n:f_{ij}(x)=0,g_{ij}(x)>0\Big\}
\]
 where $f_{ij}\!:\mathbb{R}^n\rightarrow\mathbb{R}$ and $g_{ij}\!:\mathbb{R}^n\rightarrow\mathbb{R}$
 are polynomial functions for all $1\leq i\le p,1\leq j\le q$.

 \section{Inexact PAM method for computing W-barycenter}\label{sec3}

  In this section, we shall develop an inexact PAM method for computing
  Wasserstein barycenter, which marks the main difference between D2-clustering and K-means.
  For this purpose, we first introduce the Wasserstein barycenter involved in D2-clustering.

 \subsection{Wasserstein barycenter in D2-clustering}\label{sec3.1}

 Consider discrete probability distributions with finite supports specified by
 a set of support points and their associated probabilities
 \(
   \big\{(x_1,w_1),\ldots,(x_m,w_m) \big\},
 \)
 where $x_i\in\mathbb{R}^d$ for $i=1,2,\ldots,m$ are the support vectors and $w=(w_1,\ldots,w_m)^{\mathbb{T}}\in\Delta\!:=\!\{z\in\mathbb{R}_{+}^m\,|\,\sum_{i=1}^m z_i=1\}$
 is the probability vector.
 Let $P^{\pi}\!=\!\big\{(x_i^{\pi},w_i^{\pi}),\,i=1,\ldots,m_\pi\big\}$ and
 $P^{\nu}\!=\!\big\{(x_j^{\nu},w_j^{\nu}),\,j=1,\ldots,m_\nu\big\}$ be the given
 discrete probability distributions. The $\ell_2$-Wasserstein distance between
 $P^{\pi}$ and $P^{\nu}$, denoted by $W(P^{\pi},P^{\nu})$,
 is the square root of the optimal value of the following linear programming problem
 \begin{align}\label{Wasserstein-distance}
   &W^2(P^{\pi},P^{\nu})=\min_{Z_{ij}\geq0}\sum_{i=1}^{m_\pi} \sum_{j=1}^{m_\nu}Z_{ij}\big\|x_i^{\pi} - x_j^{\nu}\big\|^2\nonumber\\
   &\qquad\qquad\qquad\quad {\rm s.t.}\ \ \sum_{j=1}^{m_\nu} Z_{ij}=w_i^{\pi},\quad i=1,2,\ldots,m_\pi;\\
   &\qquad\qquad\qquad\qquad\quad\sum_{i=1}^{m_\pi} Z_{ij}=w_j^{\nu},\quad j=1,2,\ldots,m_\nu,\nonumber
 \end{align}
 and an optimal solution of \eqref{Wasserstein-distance}
 is called  the optimal matching weights between support points $x_i^{\pi}$
 and $x_j^{\nu}$ (or the optimal coupling for $P^{\pi}$ and $P^{\nu}$).

 Given the number of clusters $K$ and a set of discrete distributions $\big\{P^{t},\,t=1,2,\ldots,N\big\}$
 where $P^{t}=\big\{(a_j^{t},b_j^{t})\in\mathbb{R}^d\times\mathbb{R},\,j=1,\ldots,n_t\big\}$,
 the goal of D2-clustering is to seek a set of centroid distributions
 $Q^*=\{Q^{s,*},\,s=1,2,\ldots,K\}$ such that
 \begin{align}\label{D2-clusteringprob}
 Q^*\in\mathop{\arg\min}_{Q^{1},\ldots,Q^{K}}\sum_{t=1}^{N}\min_{s\in\{1,\ldots,K\}}W^2(Q^{s},P^{t})
 \end{align}
 where $Q^{s}=\big\{(x_i^{s},w_i^{s})\in\mathbb{R}^d\times\mathbb{R},\,i=1,\ldots,m\big\}$
 for $s=1,\ldots,K$. Similar to $K$-means, D2-clustering achieves a desirable set of
 centroid distributions by alternately doing the two tasks:
 assigning each instance to the nearest centroid and computing the centroids.
 By Algorithm \ref{D2-clustering} in Appendix, the major computation challenge
 in each step of D2-clustering is to compute an optimal centroid distribution,
 called Wasserstein barycenter, for each cluster. Different from $K$-means,
 the optimal centroid distribution in \eqref{centroidpart} does not have a closed form.
 In fact, it is intractable since problem \eqref{centroidpart} is a nonconvex program
 in which the number of decision variables $m(1+d)+m\sum_{t=1}^N n_t$
 quickly becomes very large even for a rather small number of distributions each of
 which contains $10$ support points.
 \subsection{Inexact PAM method for computing centroid distribution}\label{sec3.2}

  Suppose that a set of discrete probability distributions $\{P^t\!: t=1,2,\ldots,N\}$ is given,
  where $P^{t}=\big\{(a_j^{t},b_j^{t})\in\mathbb{R}^d\times\mathbb{R},\,j=1,\ldots,n_t\big\}$,
  and $N$ is the sample size for computing a Wasserstein barycenter.
  Problem \eqref{centroidpart} in Algorithm \ref{D2-clustering} is to find an optimal $Q^*=\{(x_1^*,w_1^*),\ldots,(x_m^*,w_m^*)\}$ among all discrete probability distributions $Q=\{(x_1,w_1),\ldots,(x_m,w_m)\}$ such that
  \begin{equation}\label{centroid-prob}
    Q^*\in\mathop{\arg\min}_{Q}\frac{1}{N}\sum_{t=1}^N W^2(Q,P^{t}).
  \end{equation}
  Write $b^t:=(b_1^t,\ldots,b_{n_t}^t)^{\mathbb{T}}\in\mathbb{R}^{n_t}$.
  The minimization problem in \eqref{centroid-prob} actually takes the form of
  \begin{align}\label{centroid0}
   &\min_{{\mathcal{Z}}\in\mathbb{R}^{m\times n},w\in\mathbb{R}^m,x\in\mathbb{R}^{md}}
   \langle {\mathcal{Z}},N^{-1}F(x)\big\rangle\nonumber\\
   &\qquad\quad\ \ {\rm s.t.}\ \ Z^{t}e_{n_t} - w = 0,\quad t=1,2,\ldots,N,\\
   &\qquad\qquad\quad\ \ (Z^t)^{\mathbb{T}}e_m-b^t=0,\quad t=1,2,\ldots,N,\nonumber\\
   &\qquad\qquad\qquad {\mathcal{Z}}\in\mathbb{R}_{+}^{m\times n},w\in\Delta,\nonumber
 \end{align}
 where $\mathcal{Z}\!=\![Z^1\ \cdots\ Z^{N}]\in\mathbb{R}^{m\times n}$
 and $F(x):=[F^1(x)\ \cdots\ F^{N}(x)]\in\mathbb{R}^{m\times n}$
 with $n=\sum_{t=1}^N n_t$ and
 \[
    [F^t(x)]_{ij}:=\|x_i-a_j^t\|^2\quad{\rm for}\  x=(x_1;\cdots;x_m)\in\mathbb{R}^{md}.
 \]
 The LP solvers developed in \cite{Wang14,YeLi14,Yang18} are precisely
 solving the problem in \eqref{centroid0} with a fixed $x$.

 For each $t\in\{1,2,\ldots,N\}$, let
 \(
   \Sigma_t:=\big\{Y^t\in\mathbb{R}_{+}^{m\times n_t}\ |\ (Y^t)^{\mathbb{T}}e_m=b^t\big\}.
 \)
 By using the indicator functions of the sets $\Sigma_t$ and $\Delta$,
 problem \eqref{centroid0} can be compactly written as
  \begin{align}\label{centroid}
   &\min_{{\mathcal{Z}}\in\mathbb{R}^{m\times n},w\in\mathbb{R}^m,x\in\mathbb{R}^{md}}
   \sum_{t=1}^N\Big[N^{-1}\big\langle Z^t,F^t(x)\big\rangle +\delta_{\Sigma_t}(Z^t)\Big]+\delta_{\Delta}(w)\nonumber\\
   &\qquad\qquad {\rm s.t.}\ \ Z^{t}e_{n_t} -w = 0,\quad t=1,2,\ldots,N.
 \end{align}
 Although the objective function of problem \eqref{centroid} is nonconvex,
 it has a desirable coupled structure, that is, when one of the variables
 $x$ and ${\mathcal{Z}}$ is fixed, it becomes a solvable convex program.
 Inspired by this, we solve problem \eqref{centroid} in an alternating way.
 The iterate steps are described as below.
 \vspace{-0.3cm}
 \begin{algorithm}[H]
 \caption{\label{BCD}{({\bf iPAM method for solving \eqref{centroid}})}}
 \textbf{Initialization:} Choose $\alpha_0>\underline{\alpha}>0,\rho_0>\underline{\rho}>0$
              and an starting point $({\mathcal{Z}^0},w^0,x^0)$. Set $k:=0$.\\
 \textbf{while} the stopping conditions are not satisfied \textbf{do}
 \begin{itemize}
  \item[1.]  Compute
              \begin{align}\label{subprob1-Zw}
               ({\mathcal{Z}}^{k+1},w^{k+1})\approx &\mathop{\arg\min}_{Z,w}\bigg\{\sum_{t=1}^N\Big[N^{-1}\langle Z^t, F^t(x^k)\rangle+\delta_{\Sigma_t}(Z^t)\Big]+\delta_{\Delta}(w)\nonumber\\
               &\qquad\qquad\quad +\frac{\alpha_k}{2}\Big[\|{\mathcal{Z}}-{\mathcal{Z}^{k}}\|_F^2+\|w-w^k\|^2\Big]\bigg\}\\
               &\ {\rm s.t.}\ \ Z^{t}e_{n_t} - w = 0,\quad t=1,2,\ldots,N.\nonumber
              \end{align}

  \item[2.]   Compute
               \begin{equation}\label{subprob1-x}
                x^{k+1}=\mathop{\arg\min}_{x\in\mathbb{R}^{md}}\bigg\{ \frac{1}{N}\sum_{t=1}^N\bigg[\sum_{i=1}^m\sum_{j=1}^{n_t}Z_{ij}^{t,k+1}\|x_i-a_j^{t}\|^2\bigg]+\frac{\rho_k}{2}\|x-x^k\|^2\bigg\}.
               \end{equation}

  \item[3.]  Choose $\alpha_{k+1}\in[\underline{\alpha},\alpha_k]$
              and $\rho_{k+1}\in[\underline{\rho},\rho_k]$. Let $k\leftarrow k+1$, and go to Step 1.
 \end{itemize}
 \textbf{end\ while}
 \end{algorithm}
 \begin{remark}\label{remark-BCD}
  {\bf(a)} Since the term $N^{-1}\sum_{t=1}^N\big\langle Z^t,F^t(x)\big\rangle$
  in the objective function of \eqref{centroid} does not have a globally Lipschitz
  continuous gradient, we use a proximal strategy instead of a majorization technique
  as in \cite{Xu17} for each block subproblem. The proximal term $\frac{\alpha_k}{2}(\|\cdot-{\mathcal{Z}^{k}}\|_F^2+\|\cdot-w^k\|^2)$ ensures
  that a strongly convex QP instead of an LP subproblem is solved at each iteration.
  Consider that each subproblem in \eqref{subprob1-Zw} is only a convex relaxation
  to the original nonconvex problem \eqref{centroid}, and its solution with high accuracy
  may not be the best. In view of this, we seek an inexact optimal solution
  $(\mathcal{Z}^{k+1},w^{k+1})$ of each subproblem \eqref{subprob1-Zw}
  in the following sense: there exists an error matrix $\Xi^k\in\mathbb{R}^{m\times n}$
  and an error vector $\xi^k\in\mathbb{R}^m$ such that
 \begin{align}\label{inexactPAMx}
  &(\mathcal{Z}^{k+1},w^{k+1})=\mathop{\arg\min}_{Z,w}\bigg\{\sum_{t=1}^N\Big[N^{-1}\langle Z^t, F^t(x^k)\rangle+\delta_{\Sigma_t}(Z^t)\Big]-\langle(\Xi^k,\xi^k),(\mathcal{Z},w)\rangle\nonumber\\
               &\qquad\qquad\qquad\qquad\qquad\qquad +\delta_{\Delta}(w)
               +\frac{\alpha_k}{2}\Big[\|{\mathcal{Z}}-{\mathcal{Z}^{k}}\|_F^2+\|w-w^k\|^2\Big]\bigg\}\nonumber\\
               &\qquad\qquad\qquad\qquad {\rm s.t.}\ \ Z^{t}e_{n_t} - w = 0,\quad t=1,2,\ldots,N
  \end{align}
  and
  \begin{equation}\label{error-vec}
   \|\Xi^{k}\|\le\frac{\gamma_k}{2}\|\mathcal{Z}^{k+1}-\mathcal{Z}^{k}\|,\,
   \|\xi^{k}\|\le\frac{\gamma_k}{2}\|w^{k+1}-w^{k}\|\ {\rm for\ some}\ \gamma_k\in[0,\alpha_k/2].
  \end{equation}
  In Section \ref{sec5.1}, we develop a linearized ADMM for seeking such $(\mathcal{Z}^{k+1},w^{k+1})$.
  By Remark \ref{remark-sPADMM} (c) there, we know that the cost of computing
  $(\mathcal{Z}^{k+1},w^{k+1})$ is $O\big(\kappa(\sum_{t=1}^Nn_t+1)m\log{m}\big)$,
  where $\kappa\in\mathbb{N}$ is the number of iteration of the linearized ADMM for
  seeking $(\mathcal{Z}^{k+1},w^{k+1})$.

  \noindent
  {\bf(b)} Algorithm \ref{BCD} is well defined since each subproblem has a unique optimal solution.
  In particular, from the optimality condition of problem \eqref{subprob1-x},
  it is not hard to obtain
  \begin{equation}\label{xksol}
    x_i^{k+1}=\frac{2\sum_{t=1}^N\sum_{j=1}^{n_t}Z_{ij}^{t,k+1}a_j^t+\rho_kNx_i^k}
    {2\sum_{t=1}^N\sum_{j=1}^{n_t}Z_{ij}^{t,k+1}+\rho_kN }
    \ \ {\rm for}\ \ i=1,2,\ldots,m.
  \end{equation}
  The cost of computing $x^{k+1}$ is $O(md\sum_{t=1}^Nn_t)$.
  Since $m$ and $d$ are small in many cases, the main computation cost of
  Algorithm \ref{BCD} in each step is to seek an inexact solution of \eqref{subprob1-Zw}.

 \end{remark}

 To close this section, we characterize the set of the stationary points of problem \eqref{centroid}.
 Define
 \begin{equation}\label{Psi}
  \Psi({\mathcal{Z}},w,x):=f({\mathcal{Z}},w,x)+g({\mathcal{Z}},w,x)\quad\ \forall ({\mathcal{Z}},w,x)\in\mathbb{R}^{m\times n}\times\mathbb{R}^m\times\mathbb{R}^{md}
 \end{equation}
 where $f\!:\mathbb{R}^{m\times n}\times\mathbb{R}^m\times\mathbb{R}^{md}\to\mathbb{R}$
 and $g\!:\mathbb{R}^{m\times n}\times\mathbb{R}^m\times\mathbb{R}^{md}\to(-\infty,+\infty]$
 are defined by
 \begin{equation}\label{fun-fg}
  f({\mathcal{Z}},w,x)\!:=\!\frac{1}{N}\sum_{t=1}^N\big\langle Z^t,F^t(x)\big\rangle
  \ \ {\rm and}\ \
  g({\mathcal{Z}},w,x)\!:=\!\sum_{t=1}^N\big[\delta_{\Sigma_t}(Z^t)+\delta_{\Gamma_{t}}(\mathcal{Z},w)\big]+\delta_{\Delta}(w).
 \end{equation}
 Here, $\Gamma_{\!t}:=\big\{(\mathcal{Z},w)\in\mathbb{R}^{m\times n}\times\mathbb{R}^m\,|\, Z^te_{n_t}-w=0\big\}$ for each $t\in\{1,2,\ldots,N\}$.
 The following lemma provides a characterization for the set of
 the critical points of $\Psi$, which by the continuous differentiability
 of $f$ and Lemma \ref{relation} is exactly the set of directional stationary points
 of \eqref{centroid}.
 \begin{lemma}\label{CritPsi}
  The point $({\mathcal{\widehat{Z}}},\widehat{w},\widehat{x})\in{\rm crit}\Psi$
  if and only if it satisfies the following conditions
  \begin{subnumcases}{}
  \left(\begin{matrix}
   0\\ 0
   \end{matrix}\right)
   \in
   \left(\begin{matrix}
   N^{-1}F(\widehat{x})\\0
   \end{matrix}\right)
   +\left(\begin{matrix}
     \mathcal{N}_{\Sigma_1\times\cdots\times\Sigma_N}(\widehat{\mathcal{Z}})\\
      \mathcal{N}_{\Delta}(\widehat{w})
      \end{matrix}\right)+\sum_{t=1}^N\mathcal{N}_{\Gamma_{\!t}}(\widehat{\mathcal{Z}},\widehat{w}),\\
   0=2N^{-1}\sum_{t=1}^N\sum_{j=1}^{n_t}\widehat{Z}^t_{ij}(\widehat{x}_i-a_j^t).
  \end{subnumcases}
 \end{lemma}
 \begin{proof}
  Recall that $({\mathcal{\widehat{Z}}},\widehat{w},\widehat{x})\in{\rm crit}\Psi$
  if and only if $0\in\partial\Psi({\mathcal{\widehat{Z}}},\widehat{w},\widehat{x})$.
  By using \cite[Exercise 8.8]{RW98} and the smoothness of $f$,
  from \eqref{Psi} we have
  \(
    \partial\Psi({\mathcal{\widehat{Z}}},\widehat{w},\widehat{x})
    =\nabla f({\mathcal{\widehat{Z}}},\widehat{w},\widehat{x})
     +\partial g({\mathcal{\widehat{Z}}},\widehat{w},\widehat{x}).
  \)
  Notice that problem \eqref{centroid0} has a nonempty feasible set; for example,
  with $w^0=\frac{1}{m}e$ and $Z^{t,0}=\frac{1}{m}[b^t;\cdots;b^t]\in\mathbb{R}^{m\times n_t}$
  for each $t$, $({\mathcal{Z}^0},w^0,x)$ for any $x\in\mathbb{R}^{md}$ is feasible.
  Together with the polyhedrality of the sets $\Sigma_t,\Gamma_t$ and $\Delta$,
  from \cite[Theorem 23.8]{Roc70} it follows that
  \[
    \partial g({\mathcal{\widehat{Z}}},\widehat{w},\widehat{x})
    =\left(\begin{matrix}
     \mathcal{N}_{\Sigma_1\times\cdots\times\Sigma_N}(\widehat{\mathcal{Z}})\\
      \mathcal{N}_{\Delta}(\widehat{w})\\
      \{0\}^{md}
      \end{matrix}\right)+\left(\begin{matrix}
      \sum_{t=1}^N\mathcal{N}_{\Gamma_{\!t}}(\widehat{\mathcal{Z}},\widehat{w})\\
                                \{0\}^{md}
       \end{matrix}\right).
  \]
  Together with the expression of $f$ in \eqref{fun-fg},
  we obtain the desired result. \qed
 \end{proof}
 \section{Convergence analysis of Algorithm \ref{BCD}}\label{sec4}

  For the proximal alternating minimization methods, the global convergence and
  the linear convergence rate of the whole sequence have been developed in \cite{Attouch10,Attouch13,Xu13}
  under some conditions. In this section, for the inexactness $(\mathcal{Z}^{k+1},w^{k+1})$
  in the sense of \eqref{inexactPAMx}-\eqref{error-vec}, we check that the conditions
  in \cite[Section 6]{Attouch13} required by the global convergence are satisfied
  by the sequence $\{({\mathcal{Z}^k},w^k,x^k)\}$ generated by
  Algorithm \ref{BCD}, and then establish that the whole sequence converges to
  a critical point of $\Psi$. First, we study the properties of the sequence
  $\{({\mathcal{Z}^k},w^k,x^k)\}$ given by Algorithm \ref{BCD}.
 \begin{lemma}\label{decrease}
  Let $\{({\mathcal{Z}^k},w^k,x^k)\}_{k\in\mathbb{N}}$ be generated by
  Algorithm \ref{BCD} in the sense of \eqref{inexactPAMx}-\eqref{error-vec}. Then,
 \begin{itemize}
  \item[(i)] the sequence $\{\Psi({\mathcal{Z}^k},w^k,x^k)\}_{k\in\mathbb{N}}$ is nonincreasing,
               and moreover, for each $k\in\mathbb{N}$,
               \begin{align*}
                &\Psi({\mathcal{Z}^{k}},w^{k},x^{k})-\Psi({\mathcal{Z}^{k-1}},w^{k-1},x^{k-1})\\
                &\le -\frac{\alpha_{k-1}-\gamma_{k-1}}{2}\Big[\|{\mathcal{Z}^{k}}-{\mathcal{Z}^{k-1}}\|_F^2
                     +\|w^{k}-w^{k-1}\|^2\Big]-\frac{\rho_{k-1}}{2}\|x^{k}-x^{k-1}\|^2;
               \end{align*}

  \item[(ii)] $\sum_{k=1}^{\infty}\big[\|{\mathcal{Z}^{k}}-{\mathcal{Z}^{k-1}}\|_F^2+\|w^{k}-w^{k-1}\|^2+\|x^{k}-x^{k-1}\|^2\big]<\infty$,
               and consequently
               \[
                 \lim_{k\to\infty}\|{\mathcal{Z}^{k}}-{\mathcal{Z}^{k-1}}\|_F=0,\
                 \lim_{k\to\infty}\|w^{k}-w^{k-1}\|=0,\,
                 \lim_{k\to\infty}\|x^{k}-x^{k-1}\|=0;
               \]
  \item[(iii)] the sequence $\{({\mathcal{Z}^k},w^k)\}_{k\in\mathbb{N}}$ is bounded. If, in addition,
               the following level set
               \[
                 \mathcal{L}_{0}:=\big\{({\mathcal{Z}},w,x)\in\mathbb{R}^{m\times n}\times\mathbb{R}^m\times\mathbb{R}^{md}\ |\ f({\mathcal{Z}},w,x)\le f({\mathcal{Z}^0},w^0,x^0)\big\}
               \]
               is bounded where $f$ is defined by \eqref{fun-fg}, then the sequence $\{x^k\}$ is also bounded.
 \end{itemize}
 \end{lemma}
\begin{proof}
 {\bf(i)} By the definition of $({\mathcal{Z}^{k}},w^{k})$ and the feasibility of $({\mathcal{Z}^{k-1}},w^{k-1})$ to \eqref{inexactPAMx},
 \begin{align*}
   &\Psi({\mathcal{Z}^{k}},w^{k},x^{k-1})+\frac{\alpha_{k-1}}{2}\Big[\big\|{\mathcal{Z}^{k}}-{\mathcal{Z}^{k-1}}\big\|_F^2 +\|w^{k}-w^{k-1}\|^2\Big]\\
   &\leq \Psi({\mathcal{Z}^{k-1}},w^{k-1},x^{k-1})+\langle (\Xi^{k-1},\xi^{k-1}),(\mathcal{Z}^k-\mathcal{Z}^{k-1},w^k-w^{k-1})\rangle.
 \end{align*}
 Together with the inequalities in \eqref{error-vec}, it follows that
 \[
  \Psi({\mathcal{Z}^{k}},w^{k},x^{k-1})+\frac{\alpha_{k-1}-\gamma_{k-1}}{2}\Big[\big\|{\mathcal{Z}^{k}}-{\mathcal{Z}^{k-1}}\big\|_F^2 +\|w^{k}-w^{k-1}\|^2\Big]\leq \Psi({\mathcal{Z}^{k-1}},w^{k-1},x^{k-1}).
 \]
  In addition, from the definition of $x^{k}$, it immediately follows that
 \begin{equation}\label{temp-Zwxk}
    \Psi({\mathcal{Z}^{k}},w^{k},x^{k}) \leq \Psi({\mathcal{Z}^{k}},w^{k},x^{k-1})-\frac{\rho_{k-1}}{2}\|x^{k}-x^{k-1}\|^2.
 \end{equation}
 From the last two inequalities, we obtain the desired result.

 \noindent
 {\bf (ii)} From part (i) and the definition of the function $\Psi$, for each $k\in\mathbb{N}$ it holds that
 \begin{align*}
  &\frac{\alpha_{k-1}-\gamma_{k-1}}{2}\Big[\|{\mathcal{Z}^{k}}-{\mathcal{Z}^{k-1}}\|_F^2 + \|w^{k}-w^{k-1}\|^2\Big]+\frac{\rho_{k-1}}{2}\|x^{k}-x^{k-1}\|^2\\
  &\le \frac{1}{N}\sum_{t=1}^N\langle Z^{t,k-1},F^t(x^{k-1})\rangle-\frac{1}{N}\sum_{t=1}^N\langle Z^{t,k},F^t(x^{k})\rangle.
 \end{align*}
 This inequality particularly implies that for any $k'\ge 1$
 \begin{align*}
  &\sum_{k=1}^{k'}\!\Big[\frac{\alpha_{k-1}-\gamma_{k-1}}{2}\big(\|{\mathcal{Z}^{k}}-{\mathcal{Z}^{k-1}}\|_F^2
  +\|w^{k}-w^{k-1}\|^2\big)+\frac{\rho_{k-1}}{2}\|x^{k}-x^{k-1}\|^2\Big] \\
  &\le\frac{1}{N}\sum_{t=1}^N\langle Z^{t,0},F^t(x^{0})\rangle - \frac{1}{N}\sum_{t=1}^N\langle Z^{t,k'},F^t(x^{k'})\rangle
  \le\frac{1}{N}\sum_{t=1}^N\langle Z^{t,0},F^t(x^{0})\rangle.
 \end{align*}
 By taking the limit $k'\to\infty$, the desired result follows from the last inequality.

 \medskip
 \noindent
  {\bf(iii)} From the iteration steps of Algorithm \ref{BCD}, we have
 $\{Z^{t,k}\}\subseteq\Sigma_t$ for each $t=1,\ldots,N$ and $\{w^k\}\subseteq \Delta$.
 This shows that $\{({\mathcal{Z}^k},w^k)\}_{k\in\mathbb{N}}$ is bounded.
 From part (i) it follows that $\{({\mathcal{Z}^k},x^k)\}_{k\in\mathbb{N}}\subseteq\mathcal{L}_{0}$.
 Since the set $\mathcal{L}_{0}$ is bounded, $\{x^k\}$ is bounded.
 \end{proof}

 To give a subgradient lower bound for the iterate gap,
 let $U^k\!:=({\mathcal{Z}^k},w^k,x^k)$ for each $k\in\mathbb{N}$.
\begin{lemma}\label{lower-bound}
 Let $\{U^k\}_{k\in\mathbb{N}}$ be the sequence yielded by Algorithm \ref{BCD}
 in the sense of \eqref{inexactPAMx}-\eqref{error-vec}. Let
 \begin{subnumcases}{}\label{Ax}
  A^k_{\mathcal{Z}}:=\alpha_{k-1}({\mathcal{Z}^{k-1}}-{\mathcal{Z}^{k}})+N^{-1}(F(x^{k})-F(x^{k-1}))+\Xi^{k-1},\\
  A_w^{k}:=\alpha_{k-1}(w^{k-1}-w^{k})+\xi^{k-1},\\
  A_x^{k}:=\rho_{k-1}(x^{k-1}-x^{k}).
 \end{subnumcases}
 Then, for each $k\in\mathbb{N}$,
 \(
   (A^{k}_{\mathcal{Z}},A_w^k,A_x^k)\in\partial\Psi({\mathcal{Z}^k},w^k,x^k).
 \)
 If the level set $\mathcal{L}_{0}$ is bounded, then there exists an $M>0$ such that
 with $\widehat{a}=\max_{1\le t\le N,1\le j\le n_t}\|a_j^t\|$,
 \begin{equation}
   \big\|(A^k_{\mathcal{Z}},A_w^k,A_x^k)\big\|
   \le\sqrt{\max\Big(4.5\alpha_0^2,{\frac{8n}{N^2}(2M+\widehat{a})^2+\rho_0^2}\Big)}
   \big\|U^{k}-U^{k-1}\big\|.
 \end{equation}
 \end{lemma}
 \begin{proof}
  By the optimality conditions of problems \eqref{inexactPAMx} and \eqref{subprob1-x},
  it is easy to obtain
  \begin{equation*}
   \left(\begin{matrix}
    \alpha_{k-1}\!({\mathcal{Z}^{k-1}}\!-\!{\mathcal{Z}^{k}})\!-\!\frac{1}{N}\!\big(F(x^{k-1})\!-\!F(x^k)\big)\!+\!\Xi^{k-1}\!\\ \alpha_{k-1}(w^{k-1}-w^k)+\xi^{k-1}\\
    \rho_{k-1}(x^{k-1}-x^k)-\frac{2}{N}\sum_{t=1}^N\sum_{j=1}^{n_t}Z_{ij}^{t,k}(x_i^k-a_j^t)
    \end{matrix}\right)\in
    \left(\begin{matrix}
     \frac{1}{N}F(x^k)\\
              0\\
              0
     \end{matrix}\right)+ \partial g({\mathcal{Z}^k},w^k,x^k)
  \end{equation*}
  where the function $g$ is defined by \eqref{fun-fg}.
  Together with the expression of $\Psi$, we have
  \[
    (A^k_{\mathcal{Z}},A_w^k,A_x^k)\in\partial\Psi({\mathcal{Z}^k},w^k,x^k).
  \]
  From the expression of $A^k_{\mathcal{Z}}$ and the relation $\|u-v\|^2\le 2\|u\|^2+2\|v\|^2$, it follows that
  \begin{align}\label{AZk-ineq}
   \|A^k_{\mathcal{Z}}\|_F^2
   &\le 2\|\alpha_{k-1}({\mathcal{Z}^{k-1}}-{\mathcal{Z}^k})+\Xi^{k-1}\|_F^2+\frac{2}{N^2}\|F(x^{k-1})-F(x^k)\|_F^2,\nonumber\\
   &\le 4\alpha_{k-1}^2\|{\mathcal{Z}^{k-1}}-{\mathcal{Z}^k}\|_F^2+4\|\Xi^{k-1}\|_F^2+\frac{2}{N^2}\|F(x^{k-1})-F(x^k)\|_F^2,\nonumber\\
   &\le(4\alpha_{k-1}^2\!+\!\gamma_{k-1}^2)\|{\mathcal{Z}^{k-1}}\!-\!{\mathcal{Z}^k}\|_F^2
   +\frac{2}{N^2}\sum_{t=1}^N\sum_{i=1}^m\sum_{j=1}^{n_t}
   \big(\|x_i^{k-1}\!-\!a_j^t\|^2\!-\!\|x_i^{k}\!-\!a_j^t\|^2\big)^2
  \end{align}
  where the last inequality is from the first inequality in \eqref{error-vec}.
  Since $\{({\mathcal{Z}^k},w^k,x^k)\}\subseteq \mathcal{L}_{0}$ and
  the set $\mathcal{L}_{0}$ is bounded, there exists a constant
  $M>0$ such that $\|x^{k}\|\le M$ for all $k$.
  By the relation $\|u+v\|^2-\|u\|^2-\|v\|^2=2\langle u,v\rangle$,
  for each $i=1,\ldots,m$ and $j=1,\ldots,n_t$,
  \begin{align}\label{AZk-ineq1}
    \big|\|x_i^{k-1}-a_j^t\|^2-\|x_i^{k}-a_j^t\|^2\big|
    &=\big|\|x_i^k-x_i^{k-1}\|^2+2\langle x_i^k-x_i^{k-1},x_i^k-a_j^t\rangle\big|\nonumber\\
    &\le 2M\|x_i^k-x_i^{k-1}\|+2(M+\|a_j^t\|)\|x_i^k-x_i^{k-1}\|\nonumber\\
    &\le (4M+2\widehat{a})\|x_i^k-x_i^{k-1}\|.
  \end{align}
  Substituting \eqref{AZk-ineq1} into inequality \eqref{AZk-ineq} yields that
  \[
    \|A^k_{\mathcal{Z}}\|_F^2\le \max(4\alpha_{k-1}^2+\gamma_{k-1}^2,8nN^{-2}(2M+\widehat{a})^2)\big(\|{\mathcal{Z}^{k-1}}-{\mathcal{Z}^k}\|_F^2+\|x^k-x^{k-1}\|^2\big).
  \]
  Combining with the expressions of $A_w^k$ and $A_x^k$ and
  equation \eqref{error-vec} and noting that $\gamma_{k-1}\le0.5\alpha_{k-1}$,
  $\alpha_{k-1}\le\alpha_0$ and $\rho_{k-1}\le\rho_0$,
  we obtain the desired result follows. \qed
 \end{proof}

 Next we take a closer look at the KL property of the extended valued objective function $\Psi$.
 \begin{lemma}\label{PsiKL}
  The function $\Psi$ is semialgebraic, and consequently, it satisfies
  the KL property with $\phi(s)=cs^{1-\theta}$ for some $c>0$ and $\theta\in[0,1)\cap\mathbb{Q}$,
  where $\mathbb{Q}$ is the set of all rational numbers.
 \end{lemma}
 \begin{proof}
  Recall that $\Psi({\mathcal{Z}},w,x)=f({\mathcal{Z}},w,x)+g({\mathcal{Z}},w,x)$ for $({\mathcal{Z}},w,x)\in\mathbb{R}^{m\times n}\times\mathbb{R}^m\times\mathbb{R}^{md}$,
  where $f$ and $g$ are the functions defined by \eqref{fun-fg}.
  Since $g$ is an indicator on a polyhedral set which is clearly semialgebraic,
  $g$ is semialgebraic by \cite[Section 4.3]{Attouch10}. Notice that
 \[
   f({\mathcal{Z}},w,x)=\frac{1}{N}\sum_{t=1}^N\sum_{i=1}^m\sum_{j=1}^{n_t}Z_{ij}^t\|x_i-a_j^t\|^2
 \]
 is a polynomial function. So, $f$ is also semialgebraic.
 Since the sum of semialgebraic functions is semialgebraic,
 $\Psi$ is semialgebraic. The second part of the conclusions
 follows by \cite{Bolte06}. \qed
\end{proof}

 Using Lemma \ref{decrease}-\ref{PsiKL} and following the same
 arguments as those for \cite[Theorem 6.2]{Attouch13}, we can establish
 the following global convergence result of Algorithm \ref{BCD}.
  \begin{theorem}\label{main-result}
   Let $\{U^k\}_{k\in\mathbb{N}}$ be the sequence generated by Algorithm \ref{BCD}.
   Suppose that the level set $\mathcal{L}_{0}$ of $f$ is bounded. Then,
   the following assertions hold.
   \begin{itemize}
     \item [(i)] The sequence $\{U^k\}_{k\in\mathbb{N}}$ has a finite length,
                 i.e., $\sum_{k=1}^{\infty}\|U^{k+1}-U^k\|<\infty$.

     \item [(ii)] The sequence $\{U^k\}_{k\in\mathbb{N}}$ converges to a critical point
                 $\widehat{U}=(\widehat{\mathcal{Z}},\widehat{w},\widehat{x})$ of $\Psi$.
   \end{itemize}
  \end{theorem}
 \section{Numerical experiments}\label{sec5}

  We shall apply iPAM (i.e., Algorithm \ref{BCD}) to computing a Wasserstein barycenter
  in D2-clustering with unknown sparse finite supports, and compare its performance
  with that of the three-block B-ADMM (BADMM for short) proposed in \cite{YeWWL17}
  on some synthetic and real data. Notice that one may apply the state-of-art
  solver of the LP to the subproblem \eqref{subprob1-Zw} without the proximal terms.
  So, we also compare the performance of iPAM with that of such
  an alternating minimization method (abbreviated as ALMLP), for which
  the very powerful commercial package Gurobi 9.0.3 \cite{Gurobi2020}
  (with an academic license) is used to solve the LP subproblems.
  Since Gurobi is using the interior point method to solve the LPs,
  the computation cost of its each step is $O((\sum_{t=1}^Nn_t+1+mN)^3)$.
  Before doing numerical tests, we take a closer look at the solution of subproblem \eqref{subprob1-Zw}.
 \subsection{Linearized ADMM for solving subproblem \eqref{subprob1-Zw}}\label{sec5.1}

  We develop a tailored linearized ADMM for solving subproblem \eqref{subprob1-Zw},
  which is an extension of the classical ADMM designed by Glowinski and Marroco \cite{GM75}
  and Gabay and Mercier \cite{GM76}. For a given $\beta>0$, the augmented Lagrangian function of \eqref{subprob1-Zw}
 takes the following form
 \begin{align*}
   L_{\beta}({\mathcal{Z}},w;\lambda)
   &:=\sum_{t=1}^N\Big[\frac{1}{N}\langle Z^t, F^{t}(x^k)\rangle+\delta_{\Sigma_t}(Z^t)
   +\frac{\alpha_k}{2}\|Z^t-Z^{t,k}\|_F^2\Big]+\delta_{\Delta}(w)\\
   &\quad  +\sum_{t=1}^N\Big(\langle \lambda^t,Z^te-w\rangle
   +\frac{\beta}{2}\|Z^{t}e-w\|^2\Big)+\frac{\alpha_k}{2}\|w-w^k\|^2.
 \end{align*}
 With the function $L_{\beta}$, the iteration steps of the linearized ADMM are described as follows.
\begin{algorithm}[H]
 \caption{\label{semi-PADMM}{\ \ \bf Linearized ADMM for subproblem \eqref{subprob1-Zw}}}
 \textbf{Initialize:} Choose $\beta>0$ and $\tau\in(0,\frac{1+\sqrt{5}}{2})$.
  For $t=1,\ldots,N$, let $\mathcal{S}^{t,k}\!:\mathbb{R}^{m\times n_t}\to\mathbb{R}^{m\times n_t}$
  be a self-adjoint positive semidefinite linear map such that
  \(
    \alpha_k\mathcal{I}+\beta\mathcal{A}^t+\mathcal{S}^{t,k}\succeq 0,
  \)
  where $\mathcal{A}^t(X)\!:=\!Xe_{n_t}e_{n_t}^{\mathbb{T}}$ for $X\in\!\mathbb{R}^{m\times n_t}$.
  Choose an initial $(w^0,\lambda^0)\in\mathbb{R}^{m}\times\mathbb{R}^{mN}$. Set $\nu=0$.
 \begin{description}
 \item[Step 1.] Compute the following optimization problems
                \begin{subnumcases}{}\label{BA-Zpart}
                   {\mathcal{Z}^{\nu+1}}=\mathop{\arg\min}_{{\mathcal{Z}}\in\mathbb{R}^{m\times n}}
                  \bigg\{L_{\beta}({\mathcal{Z}},w^\nu;\lambda^\nu)+\frac{1}{2}\sum_{t=1}^N\|Z^t-Z^{t,\nu}\|_{\mathcal{S}^{t,k}}^2\bigg\},\\
                  \label{BA-ypart}
                  w^{\nu+1}=\mathop{\arg\min}_{w\in\mathbb{R}^m}L_{\beta}({\mathcal{Z}^{\nu+1}},w;\lambda^\nu).
                 \end{subnumcases}

  \item[Step 2.] Update the Lagrange multiplier by the formula
                 \begin{equation}\label{BA-lambdapart}
                   \lambda^{t,\nu+1}:=\lambda^{t,\nu}+\tau\beta(Z^{t,\nu+1}e_{n_t}-w^{\nu+1}),\quad t=1,2,\ldots,N.
                 \end{equation}

  \item[Step 3.] Set $\nu\leftarrow \nu+1$, and then go to Step 1.
 \end{description}
 \end{algorithm}
  \begin{remark}\label{remark-sPADMM}
  {\bf(a)} An immediate choice of $\mathcal{S}^{t,k}$ is
  \(
    \mathcal{S}^{t,k}\!:=(\sigma_{\!t}-\alpha_k)\mathcal{I}-\beta\mathcal{A}^t
  \)
  for a certain $\sigma_{\!t}\ge\alpha_k+\beta\|\mathcal{A}^t\|$.
  By the definition of $\mathcal{A}^t$,
  its spectral norm satisfies $\|\mathcal{A}^t\|\le\!\|e_{n_t}e_{n_t}^{\mathbb{T}}\|\le n_t$.
  So, $\sigma_t=\alpha_k+\beta n_t$ satisfies the requirement.
  For the subsequent numerical tests, we choose such positive semidefinite $\mathcal{S}^{t,k}$.
  For the global convergence and the linear rate of convergence of Algorithm \ref{semi-PADMM},
  the reader may refer to \cite{FPST13,HanSZ17}; and for its ergodic iteration complexity,
  the reader may refer to \cite{ShenPan16}.

  \medskip
  \noindent
  {\bf(b)} By the definition of $L_{\beta}$ and the choice of $\mathcal{S}^{t,k}$ in part (i),
  for each $t=1,\ldots,N$, it holds that
  \[
    Z^{t,\nu+1}=\mathop{\arg\min}_{Z^t\in\Sigma_t}\frac{\sigma_{\!t}}{2}\|Z^t-\sigma_t^{-1}H^t\|_F^2
  \]
  with $H^t:=\big[(\sigma_t\!-\!\alpha_k)\mathcal{I}-\beta\mathcal{A}^t\big](Z^{t,\nu})
  +\alpha_k(Z^{t,k}-(\alpha_kN)^{-1}F^t(x^k))+(\beta w^\nu-\lambda^{t,\nu})e_{n_t}^{\mathbb{T}}$, and
  \[
   w^{\nu+1}=\mathop{\arg\min}_{w\in\Delta}\frac{\beta N\!+\alpha_k}{2}\Big\|w-\frac{1}{\beta N\!+\alpha_k}
   \Big[\sum_{t=1}^N\big(\beta Z^{t,\nu+1}e+\lambda^{t,\nu}\big)+\alpha_k w^k\Big]\Big\|^2.
  \]
  The computation of ${Z^{t,\nu+1}}$ involves $n_t$ projections onto
  the simplex set $\Sigma_t$, while the computation of $\mathcal{Z}$ in
  each step involves $N$ times such projections which can be finished via
  the parallel technique. Thus,
  the computation cost of Step 1 in Algorithm \ref{semi-PADMM}
  is precisely $O\big((\sum_{t=1}^Nn_t\!+\!1)m\log{m}\big)$.
 \end{remark}

 After an elementary calculation, the dual of \eqref{subprob1-Zw} is
 the unconstrained smooth convex problem
 \begin{align}\label{dsubprob}
   &\max_{\lambda\in\mathbb{R}^{mN}}\frac{\alpha_k}{2}\bigg[\sum_{t=1}^N\big(\|\mathcal{G}_t(\lambda^t)-\Pi_{\Sigma_t}(\mathcal{G}_t(\lambda^t))\|_F^2
   -\|\mathcal{G}_t(\lambda^t)\|_F^2\big)\nonumber\\
   &\qquad\qquad\qquad +\big\|\mathcal{H}(\lambda)-\Pi_{\Delta}(\mathcal{H}(\lambda))\big\|^2
   -\|\mathcal{H}(\lambda)\|^2+M^k\bigg]
 \end{align}
 where $\lambda=(\lambda^1;\cdots;\lambda^N)\in\!\mathbb{R}^{mN},M^k=\!\sum_{t=1}^N\|Z^{t,k}\|_F^2+\|w^k\|^2$,
 $\mathcal{H}(\lambda):=w^k+\frac{1}{\alpha_k}\textstyle{\sum_{t=1}^N}\lambda^t$
 and $\mathcal{G}_t(u)\!:=Z^{t,k}\!-\!\frac{1}{\alpha_kN}F^{t}(x^k)
   -\!\frac{1}{\alpha_k}ue_{n_t}^{\mathbb{T}}$ for $t=1,\ldots,N$.
 So, during the testing, we update $\beta$ by the tradeoff
 between the primal infeasibility and relative KKT residual.
 For any $\mathcal{W}=(\mathcal{Z},w,\lambda)$, let
 \begin{align*}
  \eta_P(\mathcal{W}):=\frac{\sqrt{\sum_{t=1}^N\|Z^te_{n_t}-w\|^2}}{1+\|b\|},\ \eta_1(\mathcal{W}):=\frac{\|w-\Pi_{\Delta}(\sum_{t=1}^N\lambda^t+w-\alpha_k(w-\overline{w}))\|}{1+\|b\|},\\
  \eta_2(\mathcal{W}):=\frac{\sqrt{\sum_{t=1}^N\|Z^t-\Pi_{\Sigma_t}(Z^t-\frac{1}{N}W^t-\lambda^{t}{e_{n_t}}^{\mathbb{T}}
  -\alpha_k(Z^t-\overline{Z}^t))\|^2}}{1+\sqrt{\sum_{t=1}^N\sum_{j=1}^{n_t}\|a_j^t\|^2}}.\qquad\qquad
 \end{align*}
  It is easy to verify that $\eta(\mathcal{W})\!:=\max\{\eta_P(\mathcal{W}),\eta_1(\mathcal{W}),\eta_2(\mathcal{W})\}$
  equals $0$ if and only if $\mathcal{W}$ is a KKT point of problem \eqref{subprob1-Zw}.
  In addition, we terminate the linearized ADMM in terms of the relative KKT residual
  and the condition in \eqref{error-vec}. Specifically, for each step of Algorithm \ref{BCD},
  we terminate Algorithm \ref{semi-PADMM} whenever $\eta(\mathcal{W}^{\nu})<\epsilon_k$,
  or $\|\Xi^{\nu}\|_F<\frac{\alpha_k}{4}\|\mathcal{Z}^\nu\!-\!\mathcal{Z}^{k}\|_F$ and
  $\|\xi^{\nu}\|<\frac{\alpha_k}{4}\|w^\nu\!-\!w^{k}\|$ for $\nu\ge100$,
  where $\epsilon_k$ is updated by $\epsilon_{k+1}=\max(10^{-5},0.8\epsilon_k)$
  with $\epsilon_0=\min(\frac{12N}{\sum_{t=1}^N\!n_t},1)$.
  \vspace{-0.3cm}
 \subsection{BADMM for solving an equivalent problem of \eqref{centroid}}\label{sec5.2}

  By introducing $\mathcal{Y}=[Y^{1}\ Y^2\,\cdots\,Y^{N}]\!\in\mathbb{R}^{m\times n}$,
  problem \eqref{centroid} can be equivalently written as
  \begin{equation}\label{B-ADMM-prob}
   \min_{{\mathcal{Z}},\mathcal{Y},w,x}\left\{\sum_{t=1}^N \langle Z^{t},F^t(x)\rangle\ \ \mbox{s.t.}\
    (\mathcal{Y},w)\in\bigcap_{t=1}^N\widetilde{\Gamma}_t,\,Z^{t}=Y^{t},\,Z^{t}\in\Sigma_{t}\ {\rm for}\ t=1,\ldots,N\right\}
  \end{equation}
  where $\widetilde{\Gamma}_t=\Gamma_t\cap(\mathbb{R}_{+}^{m\times n}\times\Delta)$.
  The 3-block B-ADMM (BADMM) proposed in \cite{YeWWL17} replaces the quadratic augmented Lagrangian
  function by the Kullback-Leibler regularized Lagrange function:
  \[
   D_{\varrho}({\mathcal{Z}},\mathcal{Y},w,x;\Lambda)
   :=\sum_{t=1}^N \Big(\langle Z^{t},F^t(x)\rangle+\langle Z^{t}\!-\!Y^{t},
   \Lambda^t\rangle+\varrho D^t(Z^{t},Y^{t})\Big)
  \]
  where $\varrho>0$ is the regularization parameter, and
  $D^{t}\!:\mathbb{R}_{+}^{m\times n_t}\times\mathbb{R}_{++}^{m\times n_t}\to\mathbb{R}$
  is defined by
  \[
    D^t(Z^{t},Y^{t}):=\sum_{i=1}^m\sum_{j=1}^{n_t}Z_{ij}^{t}\Big(\log(Z_{ij}^{t}/Y_{ij}^{t})-1\Big)
    \ \ {\rm for}\ t=1,2,\ldots,N.
  \]
  Here, we stipulate $0\log 0=0$. The iteration steps of BADMM are described as follows.
 \begin{algorithm}[h]
 \caption{\label{B-ADMM}{\bf (BADMM for solving problem \eqref{B-ADMM-prob})}}
 \textbf{Initialize:} Choose $\varrho>0$ and a starting point $(Y^{t,0},w^0,x^0,\Lambda^0)$.
                      Set $k:=0$.
 \begin{description}
  \item[Step 1.] Compute the following optimization problems successively:
                \begin{subnumcases}{}\label{B-subprob-Z}
                   {\mathcal{Z}^{k+1}}=\mathop{\arg\min}_{\mathcal{Z}\in\Sigma_1\times\cdots\times\Sigma_N}D_{\varrho}(\mathcal{Z},\mathcal{Y}^{k},w^k,x^k;\Lambda^k),\\
                  \label{B-subprob-Y}
                  (\mathcal{Y}^{k+1},w^{k+1})=\mathop{\arg\min}_{(\mathcal{Y},w)\in\bigcap_{t=1}^N\widetilde{\Gamma}_t}                   \sum_{t=1}^N[\langle Z^{t,k+1}\!-\!Y^{t},\Lambda^t\rangle\!+\!\varrho D^t(Y^{t},Z^{t,k+1})],\\
                  \label{B-subprob-wx}
                  x^{k+1}\in\mathop{\arg\min}_{x\in\mathbb{R}^{md}}
                  D_{\varrho}({\mathcal{Z}^{k+1}},\mathcal{Y}^{k+1},w^{k+1},x;\Lambda^k).
                 \end{subnumcases}

  \item[Step 2.] Update the Lagrange multiplier by the formula
                 \begin{equation}\label{B-subprob-Lambda}
                   \Lambda^{k+1}:=\Lambda^k+\varrho({\mathcal{Z}^{k+1}}-\mathcal{Y}^{k+1}).
                 \end{equation}

  \item[Step 3.] Set $k\leftarrow k+1$, and then go to Step 1.
 \end{description}
 \end{algorithm}
 \begin{remark}\label{remark-BADMM}
  {\bf(a)} As discussed in \cite{YeWWL17}, subproblems \eqref{B-subprob-Z}-\eqref{B-subprob-Y}
  have a closed form solution. Among others, subproblem \eqref{B-subprob-Z} involves
  the minimization problems of the Kullback-Leibler functions over the simplex set
  $\Sigma_t$ for $t=1,\ldots,N$, and \eqref{B-subprob-Y} involves the minimization
  of the Kullback-Leibler functions on the simplex set $\widetilde{\Gamma}^t$
  for $t=1,\ldots,N$. This can be completed via the parallel technique.
  The computation cost of Step 1 in Algorithm \ref{B-ADMM}
  is $O(m\sum_{t=1}^Nn_t)+O(md\sum_{t=1}^Nn_t)$.

  \noindent
  {\bf(b)} Now it is unclear whether the BADMM is convergent or not, but
  as mentioned in the introduction, the direct extension of the classical
  ADMM to the 3-block case may be divergent.
 \end{remark}
 \subsection{Implementation details of three solvers}\label{sec5.2}

  We introduce the implementation details of iPAM, ALMLP and BADMM.
  During the testing, the mex files are written in C for the solution of
  problem \eqref{BA-Zpart}-\eqref{BA-ypart} and problem
  \eqref{B-subprob-Z}-\eqref{B-subprob-Y} so as to save the time when
  running the code in Matlab. In addition, the openmp parallel technique
  is used for the solution of problem \eqref{BA-Zpart} and problem
  \eqref{B-subprob-Z}-\eqref{B-subprob-Y}.

  For BADMM, we adopt the default setting for the parameters in the code.
  Since preliminary tests show that a varying $\rho_k$ does not
  improve the performance of Algorithm \ref{BCD}, we set $\rho_k\equiv 10^{-5}$.
  We update the parameter $\alpha_k$ by
  $\alpha_{k+1}=\max(\underline{\alpha},0.5\alpha_k)$ with
  $\alpha_0=10^2$ and $\underline{\alpha}=10^{-4}$ when
  $\frac{\alpha_k}{2}\big(\|{\mathcal{Z}^{k+1}}\!-\!{\mathcal{Z}^{k}}\|_F^2+\|w^{k+1}\!-\!w^k\|^2\big)
     >10^{-5}f({\mathcal{Z}^{k+1}},w^{k+1},x^{k+1})$, and otherwise keep unchanged.

 Notice that the KKT conditions for the nonconvex problem \eqref{B-ADMM-prob}
 takes the following form
 \begin{subnumcases}{}\label{KKTcenroid0}
   0\in N^{-1}F^t(x)+\Lambda^{t}+\mathcal{N}_{\Sigma_t}(Z^{t}),\ t=1,\ldots,N;\\
   \left(\begin{matrix}
   0\\ 0
   \end{matrix}\right)
   \in
   \left(\begin{matrix}
   -\Lambda\\ 0
   \end{matrix}\right)
   +\sum_{t=1}^N\mathcal{N}_{\widetilde{\Gamma}_{\!t}}(\mathcal{Y},w);\\
   0=\sum_{t=1}^N\sum_{j=1}^{n_t}Z_{ij}^{t}x_i-\sum_{t=1}^N\sum_{j=1}^{n_t}Z_{ij}^{t}a_j^t,\ i=1,2,\ldots,m;\\
   0=Z^{t}-Y^{t},\ t=1,\ldots,N.
 \end{subnumcases}
 We denote $\theta(\mathcal{U})\!:=\!\max\{\theta_P(\mathcal{U}),
  {\displaystyle\max_{i\in\{1,2,3\}}}\theta_i(\mathcal{U})\}$
  by its relative KKT residual at $\mathcal{U}\!=\!(\mathcal{Z},w,x,\mathcal{Y},\Lambda)$, where
 \begin{align*}
  &\theta_1(\mathcal{U})=\frac{\sqrt{\sum_{t=1}^N\|Z^t-\Pi_{\Sigma_t}(Z^t-F^t(x)-\Lambda^{t})\|_F^2}}
    {1+\sqrt{\sum_{t=1}^N\sum_{j=1}^{n_t}\|a_j^t\|^2}},\,\theta_2(\mathcal{U})=\frac{\|w-\Pi_{\Delta}(\sum_{t=1}^N\frac{1}{n_t}\Lambda^te_{n_t}+w)\|}{1+\|b\|}\\
  &\theta_3(\mathcal{U})=\frac{\sqrt{\sum_{i=1}^m\|\sum_{t=1}^N\sum_{j=1}^{n_t}Z_{ij}^{t}x_i-\sum_{t=1}^N\sum_{j=1}^{n_t}Z_{ij}^{t}a_j^t\|^2}}
    {1+\sqrt{\sum_{t=1}^N\sum_{j=1}^{n_t}\|a_j^t\|^2}},\,
    \theta_P(\mathcal{U})=\frac{\sqrt{\sum_{t=1}^N\|Z^t-Y^t\|_F^2}}{1+\sqrt{\sum_{t=1}^N\sum_{j=1}^{n_t}\|a_j^t\|^2}}.
  \end{align*}
  From the subfigure on the right side of Figure \ref{fig2}, we find that
  the residual KKT residual yielded by BADMM does not descend as the iterate steps increase.
  This means that the relative KKT residual can not be used as the stopping condition for BADMM.
  So, for the subsequent numerical tests, we terminated Algorithm \ref{B-ADMM}
  at the iterate $\mathcal{U}^k=({\mathcal{Z}^{k}},w^{k},x^{k},\mathcal{Y}^{k})$ whenever
  \[
    {\rm pinf}^{k}:=\max_{1\le t\le N}\frac{\|Z^{t,k+1}e_{n_t}-w^{k+1}\|}{1+\|b\|}\le 10^{-4}
    \ \ {\rm for}\ \ k\ge1000
  \]
  or
  \[
    \frac{\max_{0\le i\le9}|f(\mathcal{Z}^{k-i},w^{k-i},x^{k-i})-f(\mathcal{Z}^{k-i-1},w^{k-i-1},x^{k-i-1})|}
    {\max(1,f(\mathcal{Z}^{k},w^{k},x^{k}))}\le 10^{-4}\ \ {\rm for}\ k\ge 3000.
  \]

 Notice that the KKT conditions for the nonconvex problem \eqref{centroid0}
 take the following form
 \begin{subnumcases}{}\label{KKTcenroid0}
   0\in \frac{1}{N}F^t(x)+\lambda^{t}{e_{n_t}}^{\mathbb{T}}+\mathcal{N}_{\Sigma_t}(Z^{t}),\ t=1,\ldots,N;\\
   0\in -\sum_{t=1}^N\lambda^{t}+\mathcal{N}_{\Delta}(w);\\
   0=\sum_{t=1}^N\sum_{j=1}^{n_t}Z_{ij}^{t}x_i-\sum_{t=1}^N\sum_{j=1}^{n_t}Z_{ij}^{t}a_j^t,\ i=1,2,\ldots,m;\\
   0=Z^{t}e_{n_t}-w,\ t=1,\ldots,N.
 \end{subnumcases}
  Denote $\vartheta(\mathcal{V})\!:=\!\max\{\vartheta_P(\mathcal{V}),
  {\displaystyle\max_{i\in\{1,2,3\}}}\vartheta_i(\mathcal{V})\}$ by its relative KKT residual
  at $\mathcal{V}\!=\!(\mathcal{Z},w,x,\lambda)$ with
 \begin{align*}
  &\vartheta_1(\mathcal{V}):=\frac{\sqrt{\sum_{t=1}^N\|Z^t-\Pi_{\Sigma_t}(Z^t-\frac{1}{N}F^t(x)-\lambda^{t}{e_{n_t}}^{\mathbb{T}})\|_F^2}}
   {1+\sqrt{\sum_{t=1}^N\sum_{j=1}^{n_t}\|a_j^t\|^2}},
   \vartheta_2(\mathcal{V}):=\frac{\|w-\Pi_{\Delta}(\sum_{t=1}^N\lambda^t+w)\|}{1+\|b\|}\\
  &\vartheta_3(\mathcal{V}):=\frac{\sqrt{\sum_{i=1}^m\|\sum_{t=1}^N\sum_{j=1}^{n_t}Z_{ij}^{t}x_i-\sum_{t=1}^N\sum_{j=1}^{n_t}Z_{ij}^{t}a_j^t\|^2}}
  {1+\sqrt{\sum_{t=1}^N\sum_{j=1}^{n_t}\|a_j^t\|^2}},
  \vartheta_P(\mathcal{V}):=\frac{\sqrt{\sum_{t=1}^N\|Z^te_{n_t}-w\|^2}}{1+\|b\|}.
  \end{align*}
  \begin{figure}[h]
  \setlength{\abovecaptionskip}{2pt}
  \setlength{\belowcaptionskip}{0pt}
  \begin{center}
  \includegraphics[width=14cm,height=7.0cm]{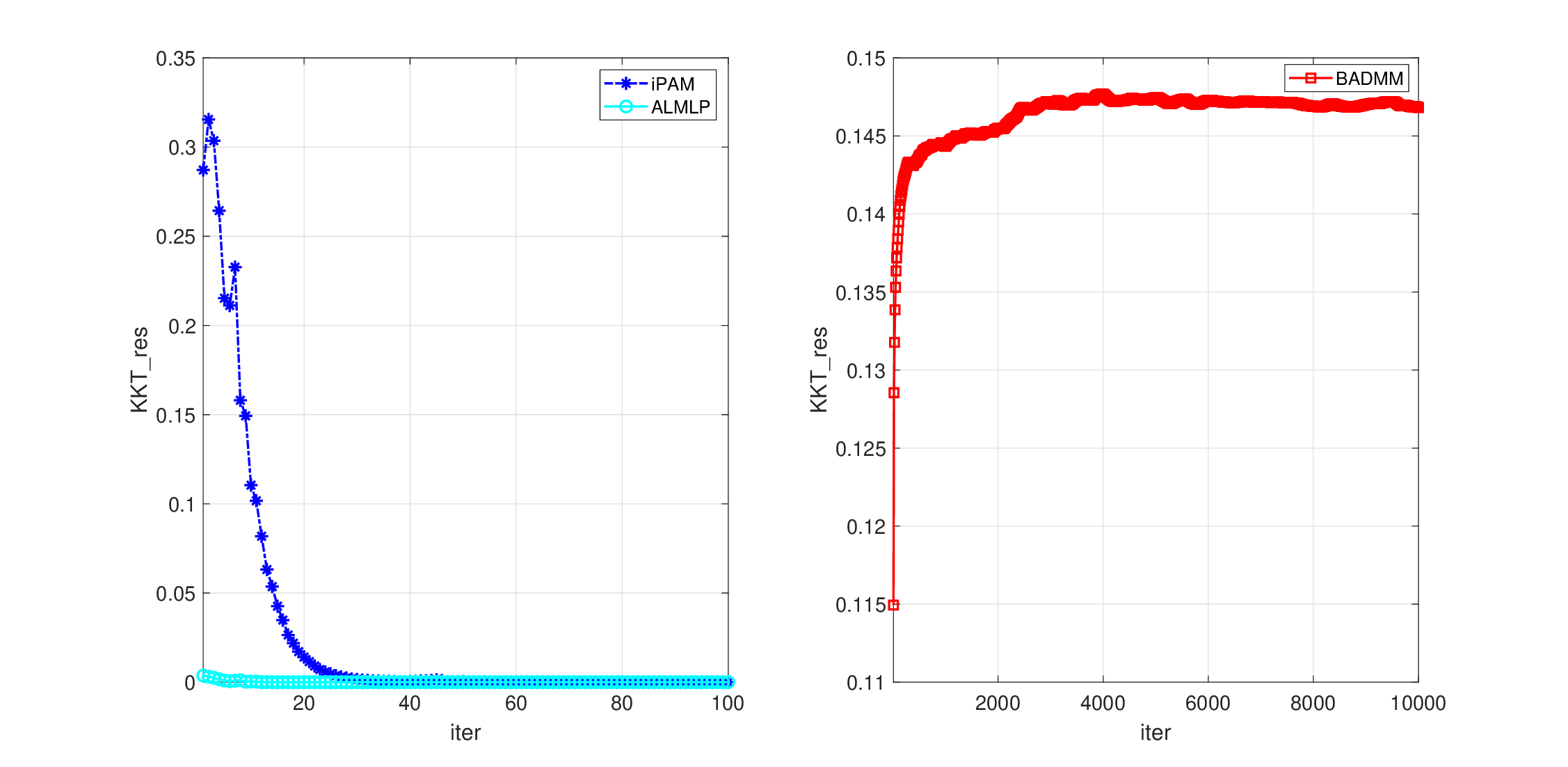}
  \end{center}
  \captionsetup{font={small}}
  \caption{The relative KKT residual curves yielded by three solvers}\label{fig2}
 \end{figure}
 The subfigure on the left side of Figure \ref{fig2} shows that
 the relative KKT residuals yielded by iPAM and ALMLP descend as the iterate
 increases. We terminate iPAM and ALMLP whenever
 $\vartheta(\mathcal{V}^k)\le 5\times 10^{-4}$ for $k\ge 5$, or
 \[
  \frac{\max_{0\le i\le9}|f(\mathcal{Z}^{k-i},w^{k-i},x^{k-i})-f(\mathcal{Z}^{k-i-1},w^{k-i-1},x^{k-i-1})|}
  {\max(1,f(\mathcal{Z}^{k},w^{k},x^{k}))}\le 10^{-4}\ {\rm for}\ k\ge 30.
 \]
  For numerical comparisons with BADMM, we also terminate iPAM when
  ${\rm pinf}^{k}\le 10^{-4}$ or
  \[
    \frac{\max_{0\le i\le9}|f(\mathcal{Z}^{k-i},w^{k-i},x^{k-i})-f(\mathcal{Z}^{k-i-1},w^{k-i-1},x^{k-i-1})|}
    {\max(1,f(\mathcal{Z}^{k},w^{k},x^{k}))}\le 10^{-4}\ \ {\rm for}\ k\ge 30.
  \]
  The above two stopping criterions are respectively named as {\bf stcond A}
  and {\bf stcond B}.

  Unless otherwise stated, for all numerical tests, the three solvers
  are using the same starting point $({\mathcal{Z}^0},w^0,x^0)$, where $\mathcal{Z}^0=0$
  and $(w^0,x^0)$ is same as the one used in the code of \cite{YeWWL17}.

 \subsection{Numerical comparisons among three solvers}\label{sec5.3}

  We shall test the performance of iPAM (i.e., Algorithm \ref{BCD} armed with Algorithm
  \ref{semi-PADMM} for solving the subproblems)\footnote{Our code can be achieved from
  \url{https://github.com/SCUT-OptGroup/Proximal_AM}} for computing Wasserstein
  barycenter of discrete probability distributions with unknown sparse finite supports
  from synthetic and real data, and compare its performance with that of
  ALMLP and BADMM.
  Among others, real data comes from USPS\footnote{\url{http://www.cs.toronto.edu/~roweis/data/usps_all.mat}},
  MNIST\footnote{\url{http://www.cs.toronto.edu/~roweis/data/mnist_all.mat}}
  and BBC News\footnote{\url{http://mlg.ucd.ie/datasets/bbc.html}}.
  Table \ref{table2} summarizes the basic information on the datasets,
  where $\overline{N}$ is the data size, $d$ is the dimension of the support vectors,
  $m$ is the number of support vectors in a barycenter. All numerical results
  are computed by a workstation running on 64-bit Windows Operating System
  with an Intel(R) Xeon(R) W-2245 CPU 3.90GHz and 128 GB RAM.
 \begin{table}[H]
 \setlength{\abovecaptionskip}{2pt}
 \setlength{\belowcaptionskip}{0pt}
 \centering
 \captionsetup{font={small}}
 \caption{Datasets used for the experiments}\label{table2}
  \begin{tabular}{c|ccccc}
  \hline\noalign{\smallskip}
  \raisebox{2.0ex}[15pt]{}
  \mbox{Data} & \quad &$\overline{N}$& $d$ & $m$ & $\frac{1}{N}\sum_{t=1}^Nn_t$\\
  \noalign{\smallskip}\hline\noalign{\smallskip}
  \raisebox{1.75ex}[0pt]{}
  \mbox{Synthetic} & \quad & 50& 2& 10& [400,3600]\\
  \noalign{\smallskip}\hline\noalign{\smallskip}
  Image color  &  \quad & 2000& 3& 60& 6\\
  \noalign{\smallskip}\hline\noalign{\smallskip}
  USPS digits  &  \quad & 11000&  2& 80& 110\\
  MNIST digits &   \quad & 10000& 2& 160& 151\\
  \noalign{\smallskip}\hline\noalign{\smallskip}
  BBC News&  \quad & 2225& 400&  25& 25\\
  \noalign{\smallskip}\hline\noalign{\smallskip}
 \end{tabular}
 \end{table}

 \paragraph{Case 1. Influence of sample size $N$ on three solvers.}
  To test the influence of $N$ on the performance of three solvers,
  we generate a set of $2000$ discrete probability distributions with
  sparse finite supports, obtained from clustering pixel colors of images as
  the paper \cite{Li08} did.

  Figure \ref{fig3} plots the average CPU time, objective value and infeasibility curves
  of iPAM, ALMLP and iPAM+LP under {\bf stcond A} for $10$ independent tests,
  and Figure \ref{fig4} plots the average CPU time, objective value, and infeasibility
  curves of iPAM, BADMM, iPAM+LP and BADMM+LP under {\bf stcond B} for $10$ independent tests.
  iPAM+LP (respectively, BADMM+LP) is same as iPAM (respectively, BADMM)
  except the problem \eqref{centroid0} with $(w,x)$ fixed as $(w^f,x^f)$ is solved with Gurobi,
  and its {\bf objval} is defined by $\frac{1}{N}\langle\mathcal{Z}^*,F(x^{f})\rangle$,
  where $(\mathcal{Z}^f,w^f,x^f)$ denotes the output of a solver
  and $\mathcal{Z}^*$ is the solution obtained by applying Gurobi to the LP
  (i.e., the problem \eqref{centroid0} with $(w,x)$ fixed as $(w^f,x^f)$).
  From Figure \ref{fig3}-\ref{fig4}, we see that iPAM requires less CPU time
  than ALMLP does under {\bf stcond A} and BADMM does under {\bf stcond B},
  and the objective values of its outputs  are remarkable superior to those
  yielded by ALMLP and a little better than those yielded by BADMM. In addition,
  the output of ALMLP has the lowest infeasibility,
  and the infeasibility yielded by iPAM is lower than that of BADMM.

 Table \ref{table3} reports the average number of iterations of iPAM and BADMM
 corresponding to Figure \ref{fig3}, where {\bf subiter} means the average total
 number of iterations of the linearized ADMM for solving subproblem \eqref{subprob1-Zw}.
 We see that the average total number of iterations of the linearized ADMM is
 much less than the average number of iterations of BADMM.
 \begin{figure}[H]
  \setlength{\abovecaptionskip}{2pt}
 \setlength{\belowcaptionskip}{0pt}
  \begin{center}
  \includegraphics[width=15cm,height=6.0cm]{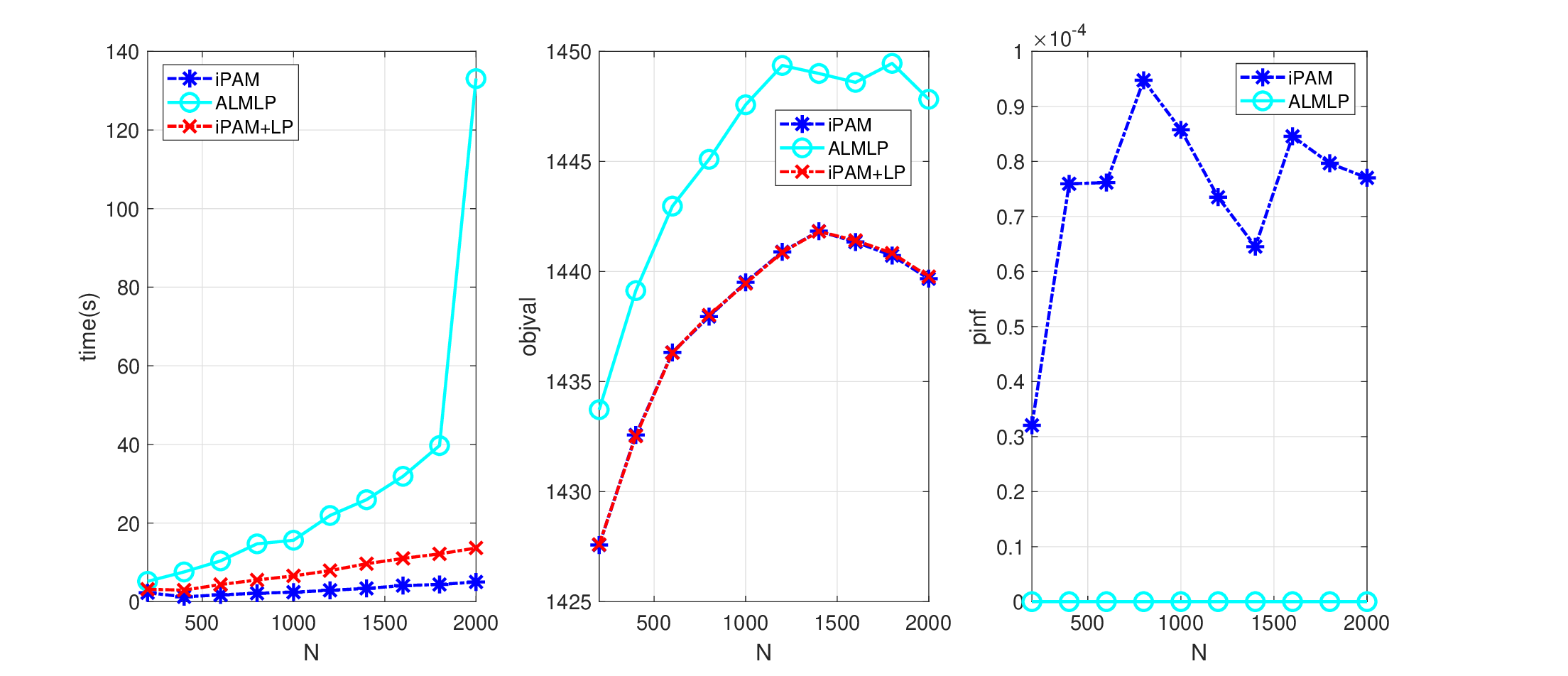}
  \end{center}
  \captionsetup{font={small}}
  \caption{Numerical comparisons among iPAM, iPAM+LP and ALMLP for different $N$ with $m=60$}
  \label{fig3}
 \end{figure}
 \vspace{-0.5cm}
 \begin{figure}[H]
  \begin{center}
  \includegraphics[width=15cm,height=6.0cm]{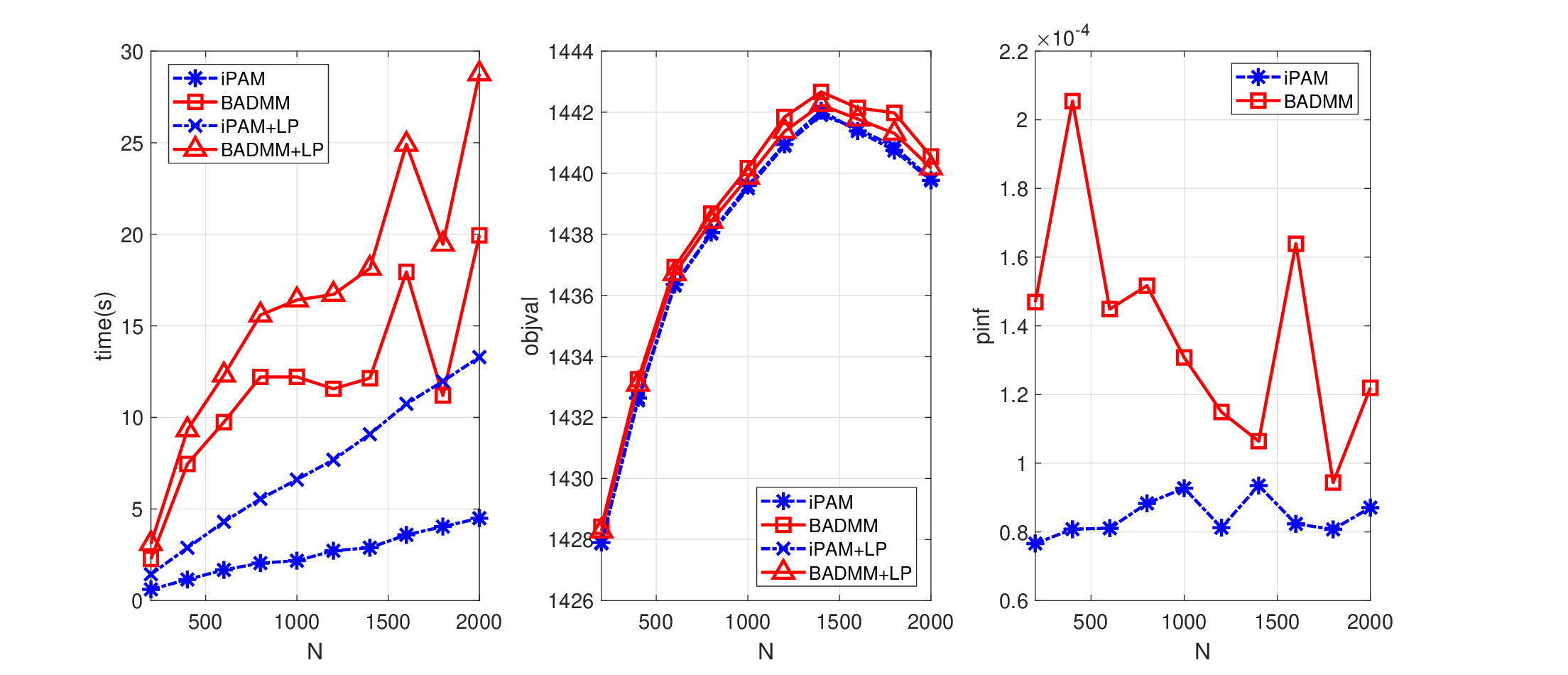}
  \end{center}
  \captionsetup{font={small}}
  \caption{Numerical comparisons among iPAM, iPAM+LP, BADMM and BADMM+LP for different $N$ with $m=60$}
  \label{fig4}
 \end{figure}
 \vspace{-0.5cm}
 \begin{table}[H]
  \setlength{\abovecaptionskip}{2pt}
 \setlength{\belowcaptionskip}{0pt}
 \caption{Average number of iterations of iPAM and BADMM corresponding to Figure \ref{fig4}}
 \label{table3}
 \tiny
 \centering
 \begin{tabular}{c|c|c|c|c|c|c|c|c|c|c}
 \hline\noalign{\smallskip}
 {\bf iter} & \multicolumn{10}{c}{$N$}\\
\cmidrule(ll){2-11}
({\bf subiter})&$200$&$400$&$600$&$800$&$1000$&$1200$&$1400$&$1600$&$1800$&$2000$\\
\noalign{\smallskip}\hline\noalign{\smallskip}
\multirow{2}*{PAM}&35.6&35.7&35.4&35.2&33.6&34.5&33.0&34.3&34.6&34.4\\
&(328.6) &(280.4) &(301.1) &(278.3) &(232.4)&(250.2)&(226.2)&(255.3)&(254.9)&(258.3) \\
\noalign{\smallskip}\hline\noalign{\smallskip}
\multirow{2}*{BADMM}& \multirow{2}*{2840}& \multirow{2}*{2720}& \multirow{2}*{2540}
& \multirow{2}*{2420}& \multirow{2}*{2040}& \multirow{2}*{1680}& \multirow{2}*{1580}& \multirow{2}*{2080}
& \multirow{2}*{1160}& \multirow{2}*{1900}\\
& & & & &&&&&&\\
\noalign{\smallskip}\hline\noalign{\smallskip}
\end{tabular}
\end{table}

 \paragraph{Case 2. Influence of number of support points $m$ on three solvers.}
  We test the influence of $m$ on the performance of three solvers by using
  the example in Case 1. Figure \ref{fig5} plots the average CPU time,
  objective value and infeasibility curves of iPAM, ALMLP and iPAM+LP under
  {\bf stcond A} for $10$ independent tests, and Figure \ref{fig6} plots
  the average CPU time, objective value and infeasibility curves of iPAM, BADMM,
  iPAM+LP and BADMM+LP under {\bf stcond B} for $10$ independent tests.
  We see that iPAM requires less CPU time than ALMLP and BADMM do,
  and the objective values of its outputs are better than those yielded by
  ALMLP and BADMM. Similarly, the infeasibility yielded by ALMLP is
  the lowest, and the infeasibility of the output of iPAM is lower than
  that of the output of BADMM.

  Table \ref{table4} reports the average number of iterations of iPAM and BADMM
  corresponding to Figure \ref{fig6}. We see that the average total number
  of iterations for the linearized ADMM is also much less than the average
  number of iterations of BADMM in this scenario.
  \vspace{-0.5cm}
  \begin{figure}[H]
  \setlength{\abovecaptionskip}{2pt}
 \setlength{\belowcaptionskip}{0pt}
  \begin{center}
  \includegraphics[width=15cm,height=6.0cm]{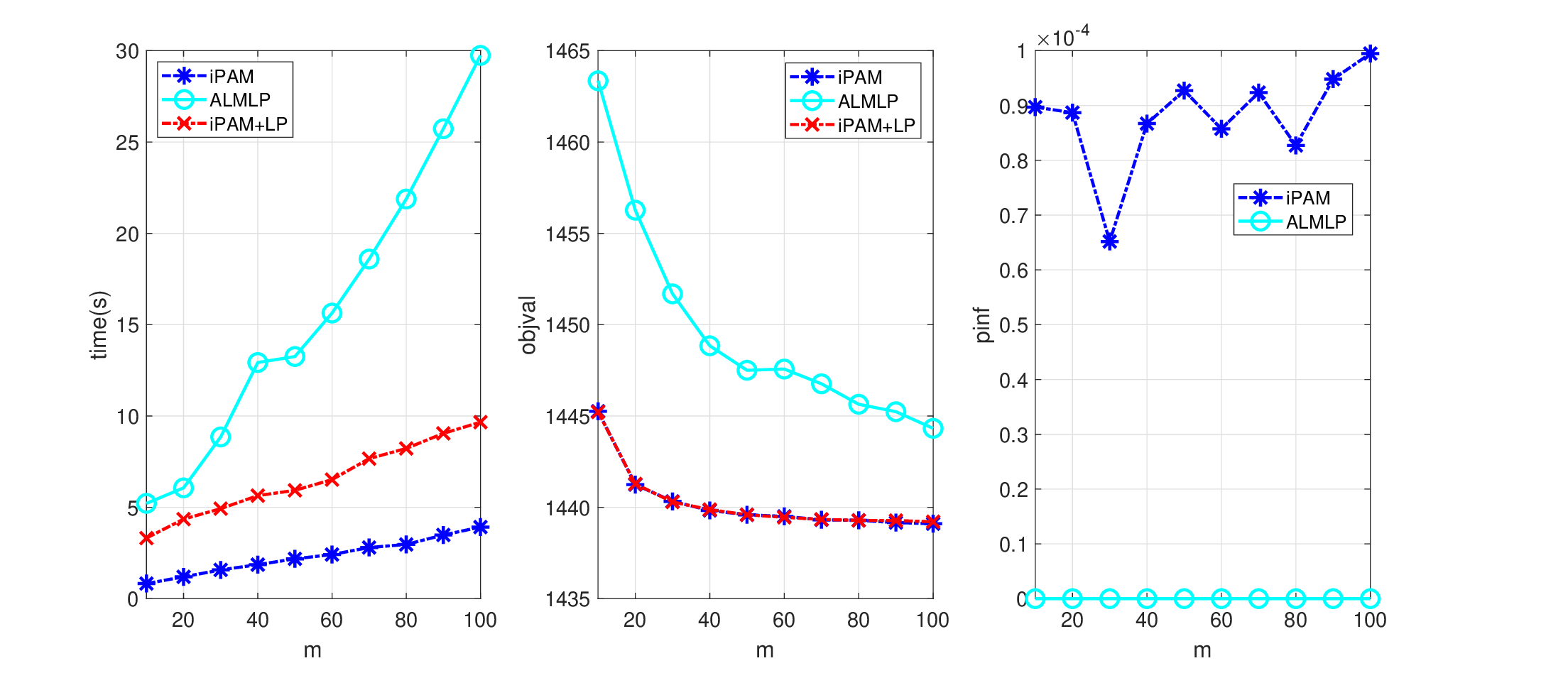}
  \end{center}
  \captionsetup{font={small}}
  \caption{Numerical comparisons among iPAM, iPAM+LP and ALMLP for different $m$ with $N=1000$}
  \label{fig5}
 \end{figure}
 \vspace{-1cm}
 \begin{figure}[H]
  \setlength{\abovecaptionskip}{2pt}
 \setlength{\belowcaptionskip}{0pt}
  \begin{center}
  \includegraphics[width=15cm,height=6.0cm]{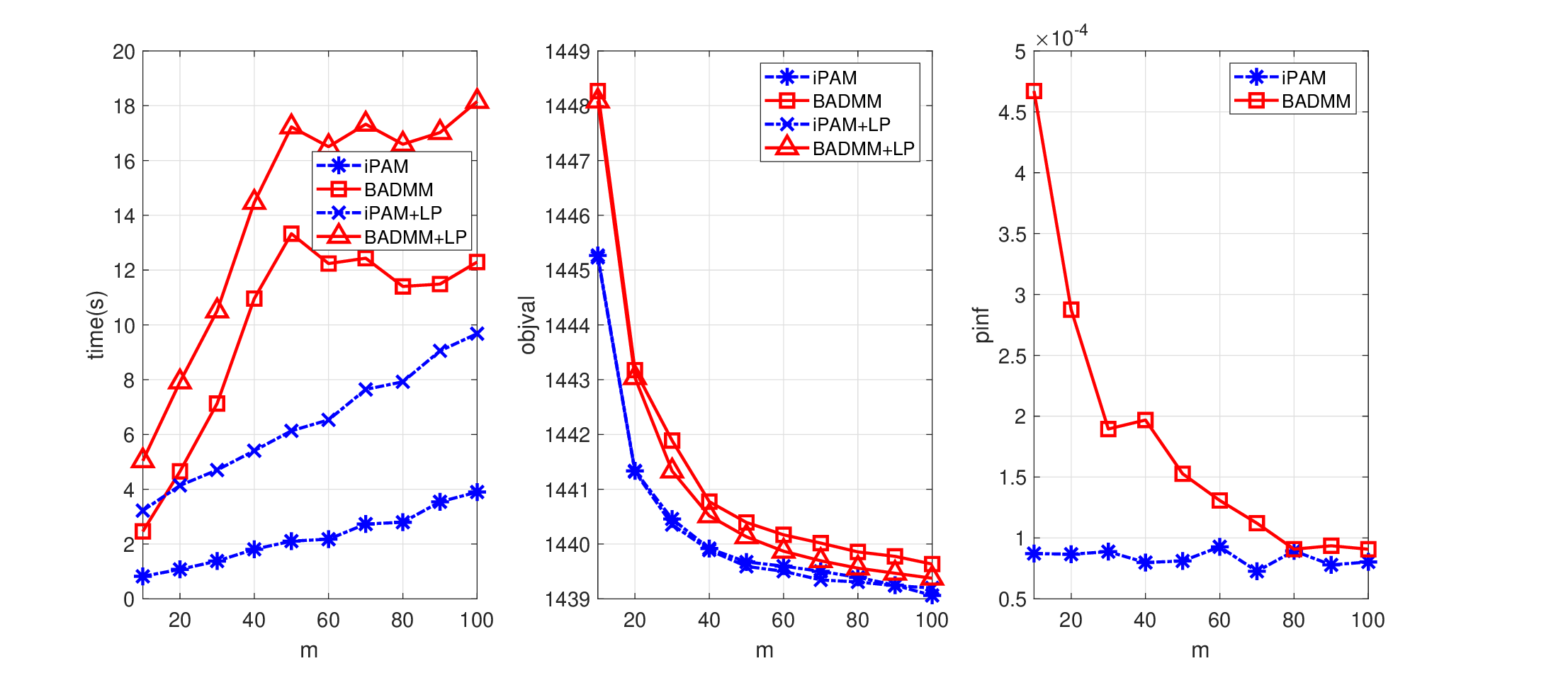}
  \end{center}
   \captionsetup{font={small}}
  \caption{Numerical comparisons of iPAM, BADMM, iPAM+LP and BADMM+LP for different $m$ with $N=1000$}
  \label{fig6}
 \end{figure}
 \begin{table}[H]
  \setlength{\abovecaptionskip}{2pt}
 \setlength{\belowcaptionskip}{0pt}
 \caption{Average number of iterations of iPAM and BADMM corresponding to Figure \ref{fig6}}
 \label{table4}
 \tiny
 \centering
 \begin{tabular}{c|c|c|c|c|c|c|c|c|c|c}
 \hline\noalign{\smallskip}
 {\bf iter} & \multicolumn{10}{c}{$m$}\\
\cmidrule(ll){2-11}
({\bf subiter})&$10$&$20$&$30$&$40$&$50$&$60$&$70$&$80$&$90$&$100$\\
\noalign{\smallskip}\hline\noalign{\smallskip}
\multirow{2}*{iPAM}&35.6&34.8&34.0&34.9&34.8&33.6&34.9&34.3&35.6&35.3\\
&(469.9) &(357.7) &(274.0) &(296.0) &(276.0)&(232.4)&(259.5)&(228.9)&(267.5)&(266.0) \\
\noalign{\smallskip}\hline\noalign{\smallskip}
\multirow{2}*{BADMM}& \multirow{2}*{3000}& \multirow{2}*{3000}& \multirow{2}*{2080}
& \multirow{2}*{2640}& \multirow{2}*{2560}& \multirow{2}*{2040}& \multirow{2}*{1840}& \multirow{2}*{1540}
& \multirow{2}*{1400}& \multirow{2}*{1380}\\
& & & & &&&&&&\\
\noalign{\smallskip}\hline\noalign{\smallskip}
\end{tabular}
\end{table}

 \paragraph{Case 3. Influence of the dimension of samples on three solvers.}
  To test the influence of $n_t$ on the performance of three solvers,
  we generate nine sets of $50$ discrete probability distributions
  with $N=50$ and $m=10$ for $n_t=400:400:3600$ in the same way
  as the paper \cite{YeWWL17} did, in which the support vectors are generated
  by sampling from a multivariate normal distribution and adding a heavy-tailed
  noise from the student's t-distribution.

  Figure \ref{fig7} plots the average CPU time, objective value and infeasibility curves
  of iPAM, ALMLP and iPAM+LP under {\bf stcond A} for $10$ independent tests,
  and Figure \ref{fig8} plots the CPU time, objective value and infeasibility curves
  of iPAM, BADMM, iPAM+LP and BADMM+LP under {\bf stcond B} for $10$ independent tests.
  We see that iPAM requires less CPU time than ALMLP and BADMM do,
  and the outputs of three solvers have comparable objective values.
  Similar to Case 1-2, the infeasibility yielded by ALMLP is the lowest,
  while that of BADMM is the highest.

  Table \ref{table5} reports the average number of iterations of iPAM and BADMM
  corresponding to Figure \ref{fig8}. We see that the average total number of
  iterations for the linearized ADMM is less than that of BADMM, although now
  it is almost three times more than that of Case 1-2.
  \vspace{-1cm}
 \begin{figure}[H]
  \setlength{\abovecaptionskip}{2pt}
 \setlength{\belowcaptionskip}{0pt}
  \begin{center}
  \includegraphics[width=15cm,height=6.0cm]{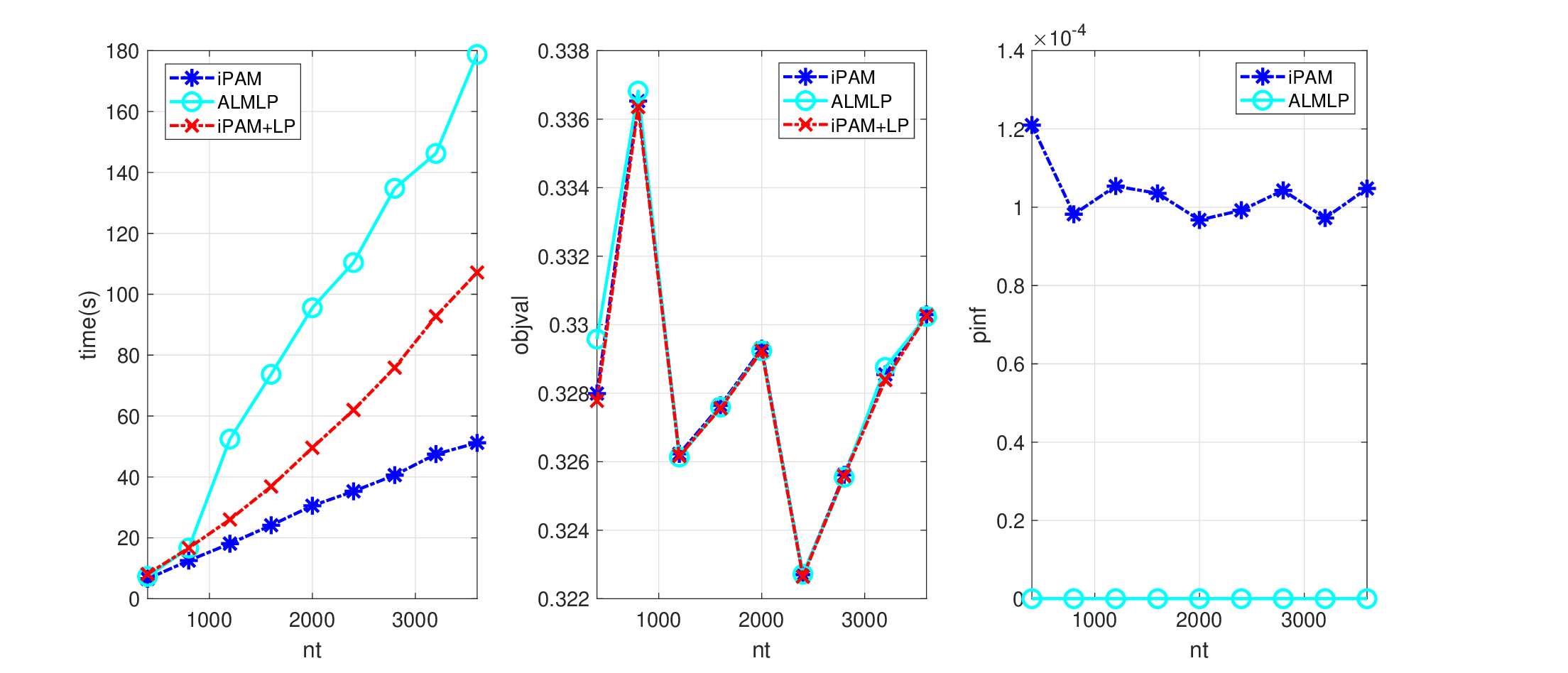}
  \end{center}
   \captionsetup{font={small}}
  \caption{Numerical comparisons among iPAM, iPAM+LP and ALMLP for different $n_t$}
  \label{fig7}
 \end{figure}
 \vspace{-0.5cm}
 \begin{figure}[H]
  \setlength{\abovecaptionskip}{2pt}
 \setlength{\belowcaptionskip}{0pt}
  \begin{center}
  \includegraphics[width=15cm,height=6.0cm]{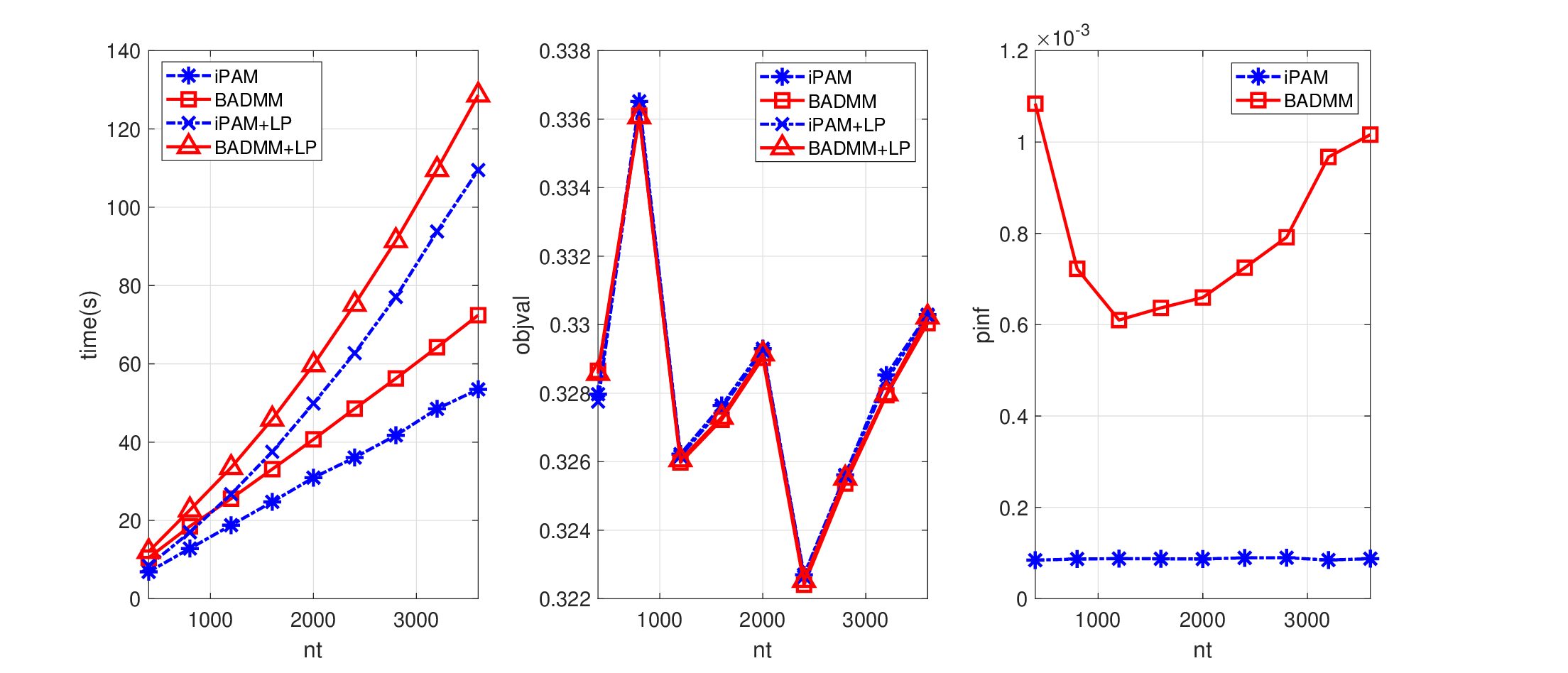}
  \end{center}
   \captionsetup{font={small}}
  \caption{Numerical comparisons among iPAM, BADMM, iPAM+LP and BADMM+LP for different $n_t$}
  \label{fig8}
 \end{figure}
 \begin{table}[h]
  \setlength{\abovecaptionskip}{2pt}
 \setlength{\belowcaptionskip}{0pt}
 \caption{Average number of iterations of iPAM and BADMM corresponding to Figure \ref{fig8}}
 \label{table5}
 \tiny
 \centering
 \begin{tabular}{c|c|c|c|c|c|c|c|c|c}
 \hline\noalign{\smallskip}
 {\bf iter} & \multicolumn{9}{c}{$n_t$}\\
\cmidrule(ll){2-10}
({\bf subiter})&$400$&$800$&$1200$&$1600$&$2000$&$2400$&$2800$&$3200$&$3600$\\
\noalign{\smallskip}\hline\noalign{\smallskip}
\multirow{2}*{PAM}&21.1&17.5&15.7&14.5&13.3&12.4&11.5&11.3&10.8\\
&(1586.5) &(1510.9) &(1479.7) &(1455.4) &(1444.2)&(1394.1)&(1377.1)&(1400.1)&(1363.7) \\
\noalign{\smallskip}\hline\noalign{\smallskip}
\multirow{2}*{BADMM}& \multirow{2}*{3000}& \multirow{2}*{3000}& \multirow{2}*{3000}
& \multirow{2}*{3000}& \multirow{2}*{3000}& \multirow{2}*{3000}& \multirow{2}*{3000}& \multirow{2}*{3000}
& \multirow{2}*{3000}\\
& & & & &&&&&\\
\noalign{\smallskip}\hline\noalign{\smallskip}
\end{tabular}
\end{table}

 \paragraph{Case 4: Numerical performance on some real data.}
 In this part, we test the performance of three solvers on some
 real data described in Example \ref{example5.3}-\ref{example5.4} below.
 \begin{example}\label{example5.3}
  We obtain a set of 2225 discrete distributions from BBC News dataset
  that is divided into five classes. The average number of support points
  is about $25$ and the dimension of every support is $400$. The texts
  are treated as a bag of words, where the support vector is
  the vocabulary of the whole document and the weight corresponds to
  the appearing frequency of words.
 \end{example}

 Table \ref{table6} reports the numerical results of three solvers
 for Example \ref{example5.3}, where the second row lists the results of iPAM
 under {\bf stcond B}. We see that iPAM yields a little better objective values
 than ALMLP does within comparable CPU time, and it also yields comparable
 even a little better objective values than BADMM does within less one fifth
 of the CPU time of the latter. The objective values yielded by iPAM+LP
 are a little better than those of iPAM.

 \begin{example}\label{example5.4}
  We obtain a set of $11000$ discrete distributions from USPS Handwritten Digits
  which is divided into ten classes and the average number of support points
  is around $110$. The digit images are treated as normalized histograms over
  the pixel locations covered by the digits, where the support vector
  is the 2D coordinate of a pixel and the weight corresponds to pixel intensity.
  \end{example}

 Table \ref{table7} reports the numerical results of three solvers
 for Example \ref{example5.4}, where the second row lists the results of iPAM
 under {\bf stcond B}. We see that iPAM yields better objective values than ALMLP does
 within much less CPU time, and it yields comparable even better objective values
 than BADMM does within less CPU time. Similar to Table \ref{table6},
 the objective values given by iPAM+LP are a little better than those yielded by iPAM.
 \begin{example}\label{example5.5}
  We obtain a set of $10000$ discrete distributions from MNIST Handwritten Digits
  that is divided into ten classes and the average number of support points is around $151$.
  The digit images are treated as normalized histograms over the pixel locations
  covered by the digits, where the support vector is the 2D coordinate of a pixel
  and the weight corresponds to pixel intensity.
  \end{example}

 Table \ref{table8} reports the numerical results of three solvers
 for Example \ref{example5.5}, where the second row lists the results of iPAM
 under {\bf stcond B}. We see that iPAM yields better objective values
 than ALMLP does, its CPU time is at least less than one fifth of the CPU time
 of ALMLP, and it also yields comparable even better objective values
 than BADMM does within less CPU time. Similarly, the objective values
 given by iPAM+LP are a little better than those given by iPAM.

 \section{Conclusions}\label{sec6}

  We have developed a globally convergent inexact PAM method for computing
  an approximate Wasserstein barycenter with unknown supports by designing
  a tailored linearized ADMM for solving the strongly convex QP subproblems.
  Numerical comparisons with the 3-block B-ADMM in \cite{YeWWL17} on synthetic
  and real data show that the proposed iPAM method has an advantage in reducing
  the computing time for large-scale problems while guaranteeing the quality
  of solutions. In our future research work, we will focus on the application of
  the iPAM method in the D2-clustering for image and document data.

\begin{acknowledgements}
 The authors would like to express their sincere thanks to Dr. Jianbo Ye
 for sharing us with their codes. The authors would like to give their
 sincere thanks to two anonymous reviewers for their comments,
 which are very helpful to improve the quality of the origin manuscript.
\end{acknowledgements}

 \bigskip
 \noindent
  {\large\bf Appendix}
 \begin{algorithm}[h]
 \caption{\label{D2-clustering}{(\bf D2-Clustering)}}
 \textbf{Initialization:} Initialize the set of centroids $\{Q^{1,0},Q^{2,0},\ldots,Q^{K,0}\}$.\\
 \textbf{For} $k=1,2,\ldots$ \textbf{do}
 \begin{itemize}
  \item[1.]  \textbf{for} $t=1,2,\ldots,N$ \textbf{do}\ \ (assignment step)
              \begin{equation}\label{P-index}
               l^{t,k}:=\mathop{\arg\min}_{s\in\{1,\ldots,K\}}W^2(Q^{s,k-1},P^{t})\qquad\qquad\qquad\qquad
               \end{equation}
              \textbf{end\ for}
  \item[2.] \textbf{for} $s=1,2,\ldots,K$ \textbf{do}\ \ (optimization step)
               \begin{equation}\label{centroidpart}
                Q^{s,k}\in\mathop{\arg\min}_{Q}\sum_{l^{t,k}=s}\!W^2(Q,P^{t})\qquad\qquad\qquad\qquad
               \end{equation}
               \textbf{end\ for}
 \end{itemize}
 \textbf{end\ For}\\
 \textbf{Return} the index set $\{l^{1,k},\ldots,l^{N,k}\}$ and the set of centroids $\{Q^{1,k},\ldots,Q^{K,k}\}$.
 \end{algorithm}
 
 \begin{sidewaystable}[h]
 \vspace{10cm}
  \setlength{\abovecaptionskip}{2pt}
  \setlength{\belowcaptionskip}{0pt}
  \centering
  \captionsetup{font={small}}
  \caption{Wasserstein Barycenters yielded by three solvers on BBC News dataset with $m=25$}\label{table6}
  \tiny
 \begin{tabular}{c|cc|cccc|cc|cccc|cc}
 \hline\noalign{\smallskip}
 &\multicolumn{2}{c|}{ALMLP} &\multicolumn{4}{c|}{iPAM}&\multicolumn{2}{c|}{iPAM+LP}&\multicolumn{4}{c|}{BADMM}&\multicolumn{2}{c}{BADMM+LP}\\
 \cmidrule(lr){2-3}\cmidrule(lr){4-7}\cmidrule(lr){8-9}\cmidrule(lr){10-13}\cmidrule(lr){14-15}
 \multicolumn{1}{c|}{\bf Class} & {\bf time(s)} & {\bf objval} & {\bf time(s)} &{\bf objval} & {\bf pinf} & {\bf iter(subiter)}& {\bf time(s)} & {\bf objval}
 & {\bf time(s)} &{\bf objval} & {\bf pinf} & {\bf iter}& {\bf time(s)} & {\bf objval}\\
 \noalign{\smallskip}\hline\noalign{\smallskip}

  \multirow{2}*{1}& \multirow{2}*{6.32}& \multirow{2}*{21.692} & 7.82 & 21.673& 1.34e-4&32(695)&10.47&21.672& \multirow{2}*{43.36}& \multirow{2}*{21.694}&\multirow{2}*{6.36e-4}&\multirow{2}*{3000}&\multirow{2}*{46.26}&\multirow{2}*{21.693}\\
  & & & 8.01 & 21.673& 7.91e-5&33(708)& 10.54 & 21.672& & & & & & \\
 \noalign{\smallskip}\hline\noalign{\smallskip}

 \multirow{2}*{2}& \multirow{2}*{11.25}& \multirow{2}*{10.336} & 6.71 & 10.331& 9.58e-5&32(850)&8.97&10.330& \multirow{2}*{33.06}& \multirow{2}*{10.329}&\multirow{2}*{4.23e-4}&\multirow{2}*{3000}&\multirow{2}*{35.44}&\multirow{2}*{10.329}\\
  & & & 6.70 & 10.331& 9.58e-5&32(850)& 8.98 & 10.330& & & & & & \\
 \noalign{\smallskip}\hline\noalign{\smallskip}

 \multirow{2}*{3}& \multirow{2}*{12.85}& \multirow{2}*{14.412} & 6.71 & 14.388& 1.45e-4&32(741)&9.07&14.388& \multirow{2}*{35.05}& \multirow{2}*{14.404}&\multirow{2}*{5.93e-4}&\multirow{2}*{3000}&\multirow{2}*{37.59}&\multirow{2}*{14.403}\\
  & & & 7.41 & 14.388& 8.50e-5&34(855)& 9.82 & 14.388& & & & & & \\
 \noalign{\smallskip}\hline\noalign{\smallskip}

 \multirow{2}*{4}& \multirow{2}*{14.73}& \multirow{2}*{9.827} & 8.44 & 9.807& 7.61e-5&32(801)&11.32&9.806& \multirow{2}*{42.69}& \multirow{2}*{9.815}&\multirow{2}*{4.88e-4}&\multirow{2}*{3000}&\multirow{2}*{46.05}&\multirow{2}*{9.814}\\
  & & & 8.44 & 9.807& 7.61e-5&32(801)& 11.47 & 9.806& & & & & & \\
 \noalign{\smallskip}\hline\noalign{\smallskip}

 \multirow{2}*{5}& \multirow{2}*{13.71}& \multirow{2}*{14.678} & 6.70 & 14.670& 6.89e-5&32(800)&8.98&14.669& \multirow{2}*{33.93}& \multirow{2}*{14.671}&\multirow{2}*{6.18e-4}&\multirow{2}*{3000}&\multirow{2}*{36.38}&\multirow{2}*{14.670}\\
  & & & 6.71 & 14.670& 6.89e-5&32(800)& 9.10 & 14.669& & & & & & \\
 \noalign{\smallskip}\hline\noalign{\smallskip}
 \end{tabular}
 \end{sidewaystable}

 \begin{sidewaystable}[h]
 \vspace{15.5cm}
  \setlength{\abovecaptionskip}{2pt}
  \setlength{\belowcaptionskip}{0pt}
  \centering
  \captionsetup{font={small}}
  \caption{Wasserstein Barycenter yielded by three solvers on USPS dataset with $m=80$}\label{table7}
  \tiny
 \begin{tabular}{c|cc|cccc|cc|cccc|cc}
 \hline\noalign{\smallskip}
 &\multicolumn{2}{c|}{ALMLP} &\multicolumn{4}{c|}{iPAM}&\multicolumn{2}{c|}{iPAM+LP}&\multicolumn{4}{c|}{BADMM}&\multicolumn{2}{c}{BADMM+LP}\\
 \cmidrule(lr){2-3}\cmidrule(lr){4-7}\cmidrule(lr){8-9}\cmidrule(lr){10-13}\cmidrule(lr){14-15}
 \multicolumn{1}{c|}{\bf Class} & {\bf time(s)} & {\bf objval} & {\bf time(s)} &{\bf objval} & {\bf pinf} & {\bf iter(subiter)}& {\bf time(s)} & {\bf objval}
 & {\bf time(s)} &{\bf objval} & {\bf pinf} & {\bf iter}& {\bf time(s)} & {\bf objval}\\
 \noalign{\smallskip}\hline\noalign{\smallskip}

 \multirow{2}*{1}& \multirow{2}*{568.15}& \multirow{2}*{4.449} & 157.54 & 4.406& 2.97e-5&27(1222)&198.31&4.404& \multirow{2}*{286.15}& \multirow{2}*{4.408}&\multirow{2}*{1.19e-4}&\multirow{2}*{3000}&\multirow{2}*{332.22}&\multirow{2}*{4.408}\\
  & & & 107.77 & 4.406& 8.34e-5&21(835)& 147.77 & 4.405& & & & & & \\
  \noalign{\smallskip}\hline\noalign{\smallskip}

  \multirow{2}*{2}& \multirow{2}*{1014.28}& \multirow{2}*{4.832} & 204.65 & 4.810& 3.24e-5&25(1275)&255.54&4.807& \multirow{2}*{358.06}& \multirow{2}*{4.807}&\multirow{2}*{1.26e-4}&\multirow{2}*{3000}&\multirow{2}*{415.53}&\multirow{2}*{4.809}\\
  & & & 142.71 & 4.812& 1.00e-4&19(885)& 193.67 & 4.809& & & & & & \\
  \noalign{\smallskip}\hline\noalign{\smallskip}

  \multirow{2}*{3}& \multirow{2}*{774.08}& \multirow{2}*{4.060} & 210.88 & 4.042& 3.17e-5&25(1290)&262.76&4.039& \multirow{2}*{366.12}& \multirow{2}*{4.038}&\multirow{2}*{1.16e-4}&\multirow{2}*{3000}&\multirow{2}*{424.77}&\multirow{2}*{4.039}\\
  & & & 158.35 & 4.044& 8.50e-5&20(965)& 210.39 & 4.040& & & & & & \\
  \noalign{\smallskip}\hline\noalign{\smallskip}

  \multirow{2}*{4}& \multirow{2}*{705.05}& \multirow{2}*{4.133} & 166.99 & 4.096& 2.96e-5&26(1204)&210.24&4.094& \multirow{2}*{248.05}& \multirow{2}*{4.098}&\multirow{2}*{9.37e-5}&\multirow{2}*{2400}&\multirow{2}*{297.57}&\multirow{2}*{4.098}\\
  & & & 111.23 & 4.098& 8.74e-5&20(799)& 154.36 & 4.095& & & & & & \\
  \noalign{\smallskip}\hline\noalign{\smallskip}

  \multirow{2}*{5}& \multirow{2}*{957.64}& \multirow{2}*{4.595} & 196.15 & 4.563& 3.89e-5&26(1313)&242.84&4.561& \multirow{2}*{290.19}& \multirow{2}*{4.567}&\multirow{2}*{9.74e-5}&\multirow{2}*{2600}&\multirow{2}*{344.38}&\multirow{2}*{4.568}\\
  & & & 143.69 & 4.565& 7.41e-5&21(963)& 190.70 & 4.562& & & & & & \\
  \noalign{\smallskip}\hline\noalign{\smallskip}
 \multirow{2}*{6}& \multirow{2}*{1463.72}& \multirow{2}*{3.455} & 206.88 & 3.426& 2.77e-5&25(1243)&259.75&3.423& \multirow{2}*{374.02}& \multirow{2}*{3.427}&\multirow{2}*{1.16e-4}&\multirow{2}*{3000}&\multirow{2}*{433.69}&\multirow{2}*{3.427}\\
  & & & 141.71 & 3.427& 8.43e-5&19(848)& 194.77 & 3.424& & & & & & \\
  \noalign{\smallskip}\hline\noalign{\smallskip}

  \multirow{2}*{7}& \multirow{2}*{763.86}& \multirow{2}*{5.232} & 163.46 & 5.204& 3.78e-5&26(1255)&203.91&5.199& \multirow{2}*{291.78}& \multirow{2}*{5.208}&\multirow{2}*{1.91e-4}&\multirow{2}*{3000}&\multirow{2}*{337.73}&\multirow{2}*{5.209}\\
  & & & 127.92 & 5.206& 8.36e-5&22(971)& 169.28 & 5.201& & & & & & \\
  \noalign{\smallskip}\hline\noalign{\smallskip}

  \multirow{2}*{8}& \multirow{2}*{1305.74}& \multirow{2}*{3.557} & 224.56 & 3.500& 2.72e-5&25(1293)&280.97&3.497& \multirow{2}*{388.02}& \multirow{2}*{3.507}&\multirow{2}*{9.76e-5}&\multirow{2}*{3000}&\multirow{2}*{453.18}&\multirow{2}*{3.506}\\
  & & & 152.55 & 3.501& 9.97e-5&19(876)& 209.39 & 3.498& & & & & & \\
  \noalign{\smallskip}\hline\noalign{\smallskip}

  \multirow{2}*{9}& \multirow{2}*{1533.53}& \multirow{2}*{3.932} & 197.28 & 3.884& 2.88e-5&25(1225)&248.87&3.881& \multirow{2}*{359.36}& \multirow{2}*{3.887}&\multirow{2}*{1.00e-4}&\multirow{2}*{3000}&\multirow{2}*{417.12}&\multirow{2}*{3.887}\\
  & & & 145.11 & 3.885& 9.60e-5&20(897)& 196.83 & 3.882& & & & & & \\
  \noalign{\smallskip}\hline\noalign{\smallskip}

  \multirow{2}*{10}& \multirow{2}*{1056.10}& \multirow{2}*{2.037} & 223.14 & 2.002& 2.85e-5&25(1237)&282.53&1.997& \multirow{2}*{244.04}& \multirow{2}*{2.011}&\multirow{2}*{9.37e-5}&\multirow{2}*{1800}&\multirow{2}*{310.09}&\multirow{2}*{2.005}\\
  & & & 137.45 & 2.003& 9.66e-5&18(758)& 196.81 & 1.999& & & & & & \\
  \noalign{\smallskip}\hline\noalign{\smallskip}
 \end{tabular}
 \end{sidewaystable}

 \begin{sidewaystable}[h]
 \vspace{13cm}
  \setlength{\abovecaptionskip}{2pt}
  \setlength{\belowcaptionskip}{0pt}
  \centering
  \captionsetup{font={small}}
  \caption{Wasserstein Barycenter yielded by three solvers on MNIST dataset with $m=160$}\label{table8}
  \tiny
 \begin{tabular}{c|cc|cccc|cc|cccc|cc}
 \hline\noalign{\smallskip}
 &\multicolumn{2}{c|}{ALMLP} &\multicolumn{4}{c|}{iPAM}&\multicolumn{2}{c|}{iPAM+LP}&\multicolumn{4}{c|}{BADMM}&\multicolumn{2}{c}{BADMM+LP}\\
 \cmidrule(lr){2-3}\cmidrule(lr){4-7}\cmidrule(lr){8-9}\cmidrule(lr){10-13}\cmidrule(lr){14-15}
 \multicolumn{1}{c|}{\bf Class} & {\bf time(s)} & {\bf objval} & {\bf time(s)} &{\bf objval} & {\bf pinf} & {\bf iter(subiter)}& {\bf time(s)} & {\bf objval}
 & {\bf time(s)} &{\bf objval} & {\bf pinf} & {\bf iter}& {\bf time(s)} & {\bf objval}\\
 \noalign{\smallskip}\hline\noalign{\smallskip}

 \multirow{2}*{1}& \multirow{2}*{4695.28}& \multirow{2}*{2.722} & 596.36 & 2.683& 2.71e-5&23(1239)&786.62&2.678& \multirow{2}*{724.51}& \multirow{2}*{2.687}&\multirow{2}*{8.16e-5}&\multirow{2}*{2200}&\multirow{2}*{922.39}&\multirow{2}*{2.681}\\
  & & & 392.29 & 2.693& 9.62e-5&17(810)& 582.02 & 2.679& & & & & & \\
  \noalign{\smallskip}\hline\noalign{\smallskip}

  \multirow{2}*{2}& \multirow{2}*{2776.96}& \multirow{2}*{2.831} & 274.07 & 2.793& 3.23e-5&27(1071)&359.58&2.793& \multirow{2}*{291.33}& \multirow{2}*{2.807}&\multirow{2}*{8.19e-5}&\multirow{2}*{1600}&\multirow{2}*{384.73}&\multirow{2}*{2.798}\\
  & & & 178.27 & 2.800& 9.46e-5&21(692)& 264.29 & 2.794& & & & & & \\
  \noalign{\smallskip}\hline\noalign{\smallskip}

  \multirow{2}*{3}& \multirow{2}*{5760.40}& \multirow{2}*{4.738} & 599.48 & 4.718& 3.19e-5&24(1326)&775.90&4.711& \multirow{2}*{885.09}& \multirow{2}*{4.714}&\multirow{2}*{7.85e-5}&\multirow{2}*{2800}&\multirow{2}*{1066.42}&\multirow{2}*{4.713}\\
  & & & 408.40 & 4.720& 9.04e-5&18(901)& 581.10 & 4.712& & & & & & \\
  \noalign{\smallskip}\hline\noalign{\smallskip}

  \multirow{2}*{4}& \multirow{2}*{4499.60}& \multirow{2}*{4.012} & 553.31 & 3.990& 3.25e-5&24(1281)&716.81&3.983& \multirow{2}*{717.42}& \multirow{2}*{3.986}&\multirow{2}*{9.53e-5}&\multirow{2}*{2400}&\multirow{2}*{891.34}&\multirow{2}*{3.986}\\
  & & & 370.61 & 3.994& 9.70e-5&18(860)& 535.51 & 3.985& & & & & & \\
  \noalign{\smallskip}\hline\noalign{\smallskip}

  \multirow{2}*{5}& \multirow{2}*{4172.91}& \multirow{2}*{3.912} & 434.40 & 3.877& 2.95e-5&24(1184)&568.99&3.873& \multirow{2}*{562.60}& \multirow{2}*{3.877}&\multirow{2}*{9.45e-5}&\multirow{2}*{2200}&\multirow{2}*{706.58}&\multirow{2}*{3.877}\\
  & & & 310.87 & 3.878& 9.06e-5&19(842)& 445.72 & 3.874& & & & & & \\
  \noalign{\smallskip}\hline\noalign{\smallskip}

  \multirow{2}*{6}& \multirow{2}*{4171.47}& \multirow{2}*{5.029} & 437.37 & 4.997& 3.25e-5&24(1214)&572.74&4.993& \multirow{2}*{557.75}& \multirow{2}*{4.995}&\multirow{2}*{9.54e-5}&\multirow{2}*{2200}&\multirow{2}*{703.30}&\multirow{2}*{4.997}\\
  & & & 317.84 & 4.998& 8.14e-5&19(878)& 452.94 & 4.994& & & & & & \\
  \noalign{\smallskip}\hline\noalign{\smallskip}

  \multirow{2}*{7}& \multirow{2}*{3914.88}& \multirow{2}*{3.672} & 516.11 & 3.647& 2.75e-5&24(1269)&670.28&3.642& \multirow{2}*{737.58}& \multirow{2}*{3.641}&\multirow{2}*{9.30e-5}&\multirow{2}*{2600}&\multirow{2}*{904.76}&\multirow{2}*{3.641}\\
  & & & 354.05 & 3.648& 8.61e-5&18(867)& 507.92 & 3.643& & & & & & \\
  \noalign{\smallskip}\hline\noalign{\smallskip}

  \multirow{2}*{8}& \multirow{2}*{5855.87}& \multirow{2}*{4.526} & 448.02 & 4.497& 3.47e-5&25(1264)&575.95&4.492& \multirow{2}*{698.46}& \multirow{2}*{4.494}&\multirow{2}*{9.16e-5}&\multirow{2}*{2800}&\multirow{2}*{835.29}&\multirow{2}*{4.491}\\
  & & & 325.18 & 4.497& 8.38e-5&20(912)& 453.29 & 4.493& & & & & & \\
  \noalign{\smallskip}\hline\noalign{\smallskip}

  \multirow{2}*{9}& \multirow{2}*{6308.66}& \multirow{2}*{3.230} & 552.55 & 3.193& 2.38e-5&24(1231)&723.67&3.189& \multirow{2}*{678.52}& \multirow{2}*{3.192}&\multirow{2}*{9.75e-5}&\multirow{2}*{2200}&\multirow{2}*{862.51}&\multirow{2}*{3.192}\\
  & & & 345.52 & 3.201& 9.47e-5&17(763)& 516.48 & 3.190& & & & & & \\
  \noalign{\smallskip}\hline\noalign{\smallskip}

  \multirow{2}*{10}& \multirow{2}*{5186.14}& \multirow{2}*{3.056} & 465.19 & 3.028& 2.70e-5&24(1205)&604.51&3.026& \multirow{2}*{643.69}& \multirow{2}*{3.032}&\multirow{2}*{9.08e-5}&\multirow{2}*{2400}&\multirow{2}*{795.67}&\multirow{2}*{3.030}\\
  & & & 312.19 & 3.035& 9.99e-5&18(807)& 453.43 & 3.027& & & & & & \\
  \noalign{\smallskip}\hline\noalign{\smallskip}
 \end{tabular}
 \end{sidewaystable}


\begin{thebibliography}{}
 \bibitem{Attouch10}
  {\sc H.\ Attouch, J.\ Bolte, P.\ Redont and A.\ Soubeyran},
  {\em Proximal alternating minimization and projection methods for nonconvex problems: an approach
  based on the Kerdyka-{\L}ojasiewicz inequality},
  Mathematics of Operations Research, 35(2010): 438-457.

  \bibitem{Attouch13}
  {\sc H.\ Attouch, J.\ Bolte and B. F.\ Svaiter},
  {\em Convergence of descent methods for semi-algebraic and tame problems:
  proximal algorithms, forward-backward splitting, and reguarlized Gauss-Seidel methods},
  Mathematical Programming, 137(2013): 91-129.

   \bibitem{Bolte06}
   {\sc J.\ Bolte,  A.\ Daniilidis and A.\ Lewis},
   {\em The Lojasiewicz inequality for nonsmooth subanalytic functions with applications to subgradient dynamical systems},
   SIAM Journal on Optimization, 17(2006): 1205-1223.

  \bibitem{Bolte14}
   {\sc J.\ Bolte, S.\ Sabach and M.\ Teboulle},
   {\em Proximal alternating linearized minimization for nonconvex and nonsmooth problems},
   Mathematical Programming, 146(2014): 459-494.

  \bibitem{Bregman67}
   {\sc L. M.\ Bregman},
   {\em The relaxation method of finding the common point of convex sets
   and its application to the solution of problems in convex programming},
   USSR Comput. Math. Math. Phys., 7(1967): 200-217.


  \bibitem{Chen16}
  {\sc C. H.\ Chen, B. S.\ He, Y. Y.\ Ye and X. M.\ Yuan},
  {\em The direct extension of ADMM for multi-block convex minimization problems
  is not necessarily convergent}, Mathematical Programming, 155(2016): 57-79.

  \bibitem{Cui13}
  {\sc M.\ Cuturi},
  {\em Sinkhorn distances: Lightspeed computation of optimal transport},
  in Proc. Adv. Neural Inf. Process. Syst., 2013: 2292-2300.

   \bibitem{Cui14}
  {\sc M.\ Cuturi and A.\ Doucet},
  {\em Fast computation ofWasserstein barycenters},
  in Proc. Int. Conf. Mach. Learn., 2014, pp. 685-693.

  \bibitem{FPST13}
 {\sc M.\ Fazel, T. K.\ Pong, D. F.\ Sun and P.\ Tseng},
 {\em Hankel matrix rank minimization with applications to system identification and realization},
 SIAM Journal on Matrix Analysis, 34(2013): 946-977.


  \bibitem {GM75}
  {\sc R.\ Glowinski and A.\ Marrocco},
  {\em Sur l' approximation par \'{e}l\'{e}ments finis d'ordre un, etla r\'{e}solution,
  par p\'{e}nalisation-dualit\'{e}, d'une classe de probl\`{e}mes de dirichlet non lin\'{e}ares},
  Revue Francaise d' Automatique, Informatique et Recherche Op\'{e}rationelle,
  9(1975): 41-76.

  \bibitem{GM76}
  {\sc D.\ Gabay and B.\ Mercier},
  {\em A dual algorithm for the solution of nonlinear variational problems via finite element approximation},
  Computers and Mathematics with Applications, 2(1976): 17-40.

  \bibitem{Gurobi2020}
  {\em Inc. Gurobi Optimization},
  Gurobi Optimizer Reference Manual, 2020.

  \bibitem{HanSZ17}
  {\sc D. R.\ Han, D. F.\ Sun and L. W.\ Zhang},
  {\em Linear rate convergence of the alternating direction method of multipliers
  for convex composite programming},
  Mathematics of Operations Research, 43(2017): 622-637.

  \bibitem{Hong16}
  {\sc M.\ Hong, M.\ Razaviyayn, Z. Q.\ Luo and J. S.\ Pang},
  {\em A unified algorithmic framework for block-structured optimization involving big data:
   with applications in machine learning and signal processing},
   IEEE Signal Processing Magazine, 33(2016): 57-77.

  \bibitem{Li08}
  {\sc J.\ Li and J. Z.\ Wang},
  {\em Real-time computerized annotation of pictures},
  IEEE Transactions on Pattern Analysis and Machine Intelligence,
  30(2008): 985-1002.

  \bibitem{Mallows72}
  {\sc C. L.\ Mallows},
  {\em A note on asymptotic joint normality},
  The Annals of Mathematical Statistics, 43(1972): 508-515.



  \bibitem{Pele09}
  {\sc O.\ Pele and M.\ Werman},
  {\em Fast and robust Earth mover's distances},
  in Proceedings of IEEE International Conference on Computer Vision,
  2009, pp. 460-467.


  \bibitem{Razaviyayn13}
  {\sc M.\ Razaviyayn, M.\ Hong and Z. Q.\ Luo},
  {\em A unified convergence analysis of block successive minimization methods for nonsmooth optimization},
  SIAM Journal on Optimization, 23(2013): 1126-1153.

  \bibitem{Roc70}
  {\sc R. T.\ Rockafellar},
  {\em Convex Analysis}, Princeton University Press, 1970.

%

  \bibitem{RW98}
 {\sc R. T.\ Rockafellar and R. J-B.\ Wets},
 {\em Variational Analysis}, Springer, 1998.


 \bibitem{Rubner00}
 {\sc Y.\ Rubner, C.\ Tomasi and L. J.\ Guibas},
 {\em The Earth mover's distance as a metric for image retrieval},
 International journal of computer vision, 40(2000): 99-121.

 \bibitem{ShenPan16}
 {\sc L.\ Shen and S. H.\ Pan},
 {\em Weighted iteration complexity of the sPADMM on the KKT residuals for convex composite optimization},
 arXiv:1611.03167.



  \bibitem{Tseng01}
  {\sc P.\ Tseng},
  {\em Convergence of a block coordinate descent method for nondifferentiable minimization},
  Journal of Optimization Theory and Application, 109(2001): 475-494.

 \bibitem{Tseng09}
  {\sc P.\ Tseng and S. W.\ Yun},
  {\em A coordinate gradient descent method for nonsmooth separable minimization},
  Mathematical Programming, 117(2009): 387-423.

  \bibitem{Villani08}
  {\sc C.\ Villani},
  {\em Optimal Transport: Old and New},
  New York, NY, USA: Springer, 2008, vol. 338.

  \bibitem{Wang14}
  {\sc H.\ Wang and A.\ Banerjee},
  {\em Bregman alternating direction method of multipliers},
  in Proc. Adv. Neural Inf. Process. Syst., 2014, pp. 2816-2824.

  \bibitem{Xu17}
  {\sc Y. Y.\ Xu and W. T.\ Yin},
  {\em A globally convergent algorithm for nonconvex optimization based on block coordinate update},
  Journal of Scientific Computing, 72(2017): 700-734.

  \bibitem{Xu13}
  {\sc Y. Y.\ Xu and W. T.\ Yin},
  {\em A block coordinate descent method for regularized multiconvex optimization with applications
  to nonnegative tensor factorization and completion},
  SIAM Journal on Imaging Sciences, 6(2013): 1758-1789.


  \bibitem{Yang18}
  {\sc L.\ Yang, J.\ Li, D. F.\ Sun and K. C.\ Toh},
  {\em A fast globally linearly convergent algorithm for the computation of Wasserstein Barycenters},
  arXiv:1809.04249.


  \bibitem{YeLi14}
  {\sc J. B.\ Ye and J.\ Li},
  {\em Scaling up discrete distribution clustering using ADMM},
  in Proceedings of International Conference on Image Process, 2014, pp. 5267-5271.

  \bibitem{YeWWL17}
  {\sc J. B.\ Ye, P. R.\ Wu, J. Z.\ Wang and J.\ Li},
  {\em Fast discrete distribution clustering using wasserstein barycenter with sparse support},
  IEEE Transactions on Signal Processing, 65(2017): 2317-2332.


  \bibitem{Zhang15}
  {\sc Y.\ Zhang, J. Z.\ Wang and J.\ Li},
  {\em Parallel massive clustering of discrete distributions},
  ACM Transactions on Multimedia Computing, Communications and Applications,
  11(2015): 49:1-49:24.


%

%
%
\end{thebibliography}
\end{document}